\documentclass[10pt]{article}

\usepackage[utf8]{inputenc}
\usepackage[margin=1.3in]{geometry}
\usepackage[titletoc,title]{appendix}
\usepackage{amsthm}
\usepackage{mathtools}
\usepackage{mathrsfs}
\usepackage{amsbsy}
\usepackage{amsfonts}
\usepackage{amsopn}
\usepackage{amscd}
\usepackage{amsxtra}
\usepackage{dsfont}
\usepackage{esint}
\usepackage{authblk}
\usepackage{hyperref}
\usepackage{xcolor}

\numberwithin{equation}{section}
\theoremstyle{plain}

\newtheorem{defin}{Definition}[section]
\newtheorem{lemma}[defin]{Lemma}
\newtheorem{prop}[defin]{Proposition}

\newtheorem{teorema}[defin]{Theorem}
\newtheorem*{teorema*}{Theorem}
\theoremstyle{definition}

\newtheorem{oss}[defin]{Remark}

\newcommand{\bd}{\partial}

\newcommand{\R}{\mathbb{R}}

\newcommand{\N}{\mathbb{N}}
\newcommand{\Z}{\mathbb{Z}}
\newcommand{\T}{\mathbb{T}}

\newcommand{\ud}{\,\textnormal{d}}
\newcommand{\dist}{\textnormal{dist}}
\newcommand{\virg}[1]{``#1''}
\let\div\undefined
\newcommand{\div}{\textnormal{div}}
\newcommand{\udH}{\ud\mathcal H^{N-1}}
\let\dist\undefined
\newcommand{\dist}{\textnormal{dist}}
\newcommand{\sd}{\textnormal{sd}}
\newcommand\rosso[1]{{\textcolor{black}{#1}}}

\begin{document}

\title{Long Time Behaviour of the Discrete Volume Preserving Mean Curvature Flow in the Flat Torus}
\author[1]{Daniele De Gennaro}
\author[2]{Anna Kubin}
\affil[1]{CEREMADE department, Université Paris IX - Dauphine, pl. du Maréchal de Lattre de Tassigny, 75775 Paris Cedex 1. Email address: degennaro@ceremade.dauphine.fr}
\affil[2]{Dipartimento di Scienze Matematiche “G.L.~Lagrange”, Politecnico di Torino, c.so Duca degli Abruzzi 24, 10129 Torino, Italy. Email address: anna.kubin@polito.it}

\date{}

\maketitle
\begin{abstract}
	\noindent We show that the discrete approximate volume preserving mean curvature flow in the flat torus $\T^N$ starting near a strictly stable critical set $E$ of the perimeter converges in the long time to a translate of $E$ exponentially fast. As an intermediate result we establish a new quantitative estimate of Alexandrov type for periodic strictly stable constant mean curvature hypersurfaces. Finally, in the two dimensional case a complete characterization of the long time behaviour of the discrete flow with  arbitrary initial sets of finite perimeter is provided.
\end{abstract}

\tableofcontents


\section{Introduction}
We consider the geometric evolution of sets called \textit{the volume preserving mean curvature flow}. The \textit{classical mean curvature flow} is defined as a flow of sets $(E_t)_{0\leq t\leq T}$ in $\R^N$ following the motion law
\begin{equation*}
	v_t=-H_{E_t}\quad \text{on} \quad \bd E_t,
\end{equation*}
where $v_t$ denotes the component of the velocity relative to the outer normal vector of $\bd E_t$  and $H_E$ is the mean curvature of the set $E$.
In order to include the volume constraint, one can consider the following velocity
\begin{equation}
	v_t= \overline{H}_{E_t}- H_{E_t}\quad \text{on} \quad \bd E_t
	\label{legge evoluzione}
\end{equation}
for all $t\in[0,T]$ , where $\overline{H}_{E_t}$ denotes the average of $H_{E_t}$ over $\bd E_t$.

The defined geometric evolution is called \textit{volume constrained mean curvature flow} or \textit{volume preserving mean curvature flow}.  One can observe that the volume of the evolving sets is indeed preserved during the evolution and that the perimeters of the sets $E_t$ are non-increasing.

This geometric flow has been used to describe some types of solidification processes and coarsening phenomena in physical systems. For example, one can consider mixtures that, after a first relaxation time, can be described by two subdomains of nearly pure phases far from equilibrium, evolving in a way to minimize the total interfacial area between the phases while keeping their volume constant (further details on the physical background can be found in \cite{CRCT,MS}, see also the introduction of \cite{MPS}).
Moreover, some variants of the volume-preserving mean curvature flow were also applied in the context of shape recovery in image processing in \cite{CD}.

One of the main mathematical difficulties of the volume preserving mean curvature flow is the non-local nature of the functional given by the constraint. Moreover, the generated flow may present singularities of different kinds, even in a finite time-span and even if the initial data is smooth. For example, we can see merging or collision of near sets, pinch-offs or shrinking of connected components to points. There exist examples of singular solutions even in the two dimensional case, see \cite{May01,MaSi}. After the onset of singularities,
the classical or smooth formulation of the flow \eqref{legge evoluzione} ceases to hold and needs to be replaced by a weaker one.
Due to the lack of a comparison principle, a natural approach is the minimizing movement approach proposed independently by Almgren, Taylor and Wang in \cite{ATW} and by Luckhaus and Sturzenhecker in \cite{LS} for the unconstrained case and adapted to the volume-preserving setting in \cite{MSS}.

We briefly recall the scheme in the volume contrained setting.
First of all we define a discrete-in-time approximation of the flow that will be called the \textit{discrete (volume-preserving) flow}. Given any initial set $E_0$ and a time-step $h>0$ we define iteratively $E_h^0:=E_0$ and for all $n\ge 0$
\[ E_h^{n+1}\in \text{argmin}\left\lbrace P(F)+\dfrac 1h \int_{F\triangle E_h^n} \dist_{\bd E_h^n}(x)\ud x\ :\ F\subset \T^N,\,|F|=|E_0| \right\rbrace,  \]
where $\dist_{\bd E_h^n}$ is the distance function from the set $\bd E_h^n$. We can define for every $t \ge 0$, the approximate flow by $E_h(t):=E_h^{[t/h]}$.
It can be proved (see \cite[Proposition~2.2]{MPS} ) that the discrete flow is well defined. 
Any limit point of this flow as the time-step $h$ converges to zero will be called a \textit{flat flow}. 
As for the classical mean curvature flow, this approach produces global-in-time solutions as shown in \cite{MSS}. The existence of such global solutions then permits to analyse the equilibrium configurations reached in the long time asymptotics.

The long time behaviour of the volume preserving mean curvature flow has been previously studied only in some particular cases, when the existence of global smooth solutions could be ensured by choosing suitably regular initial sets. For example one can consider uniformly convex and nearly spherical initial sets (see \cite{ES,Hui}), or $C^\infty-$regular initial sets that are $H^3-$close to strictly stable critical sets in the three and four dimensional flat torus (see \cite{Nii}). 
For more general initial data, the long time behaviour in the context of flat flows of convex and star-shaped sets (see \cite{BCCN, KiKU}) has been characterized only up to (possibly diverging in the case of \cite{BCCN}) translations. 
In \cite{MPS} the authors characterized the long-time limits of the discrete-in-time approximate flows constructed by the Euler implicit scheme introduced in \cite{ATW, LS} under the volume constraint in arbitrary space dimension. They proved that the discrete flow starting from an arbitrary bounded initial set converges exponentially fast to a finite union of disjoint balls with equal radii. The same authors and collaborators were also able to send the discretization parameter $h$ to 0 in the preprint~\cite{JMPS}, in the case $N=2$.
Indeed, an explicit penalization is used in order to enforce the volume constraint.

In this paper the long-time convergence analysis is developed in the flat torus $\T^N$ for the discrete flow. In such framework the class of possible long-time limits is much richer as it includes not only union of  balls with equal radii but also different type of critical sets for the perimeter. 
The notion of \textit{strictly stable critical set} is crucial to our result; for the precise definition we refer to Section~\ref{preliminaries}, but it can be summarized as a regular, critical set for the perimeter (i.e. with a constant mean curvature boundary) with strictly positive (volume-constrained) second variation. The main result of the paper is the theorem below. It provides a complete characterization of the long-time behaviour of the discrete mean curvature flow in the flat torus starting near a strictly stable critical set. Moreover, an estimate on the convergence speed is provided. 
\begin{teorema}
	Let $E$ be a strictly stable critical set in the flat torus. Then there exist $\delta^*=\delta^*(E)>0$ and $h^*=h^*(E)>0$  with the following property: if $h< h^*$ and $E_0\subset \T^N$ is a set of finite perimeter satisfying
	\[|E_0| =|E|, \qquad \overline{E}_0\subset (E)_{\delta^*},\]
	then every discrete volume  preserving mean curvature flow $(E_h^n)_{n\in\N}$ starting from $E_0$ converges to a translate of $E$ in $C^k$ for every $k\in\N$ and the convergence is exponentially fast.
	\label{teorema convergenza esponenziale}
\end{teorema}
We would like to give some details to highlight the major differences between our results and the analysis carried out in \cite{Nii}. In the aforementioned work, the author studied the flat flow, albeit in low dimension ($N\le 4$). In the article, it was assumed the initial set to be a $C^\infty-$deformation of a strictly stable critical set, close in the $H^3-$sense to the latter set. 
Under these assumptions, it was proved the exponential convergence of the flat flow to a translated of the strictly stable critical set. We remark that our result addresses the long time behaviour of the discrete flow but holds in much weaker hypotheses: we only assume the initial set to be of finite perimeter and close in the Hausdorff sense to a strictly stable critical set. Moreover, our result holds in every dimension and we are also able to provide the complete characterization of the long-time behaviour starting from any initial set in dimension $N=2$. In order to state the precise result in the two-dimensional case we first introduce the following notation. 

We will call \rosso{\textit{lamella} any connected set in $\T^2$ whose $1-$periodic extension in $R^2$ is a  stripe bounded by two parallel lines.} Our final result in two dimension is the following theorem.
\begin{teorema}
	Fix $h$, $m>0$ and an initial set $ E_0 \subset \T^2$ with finite perimeter and such that $|E_0|=m$.
	Let $(E_h^n)_{n \in \N}$ be a discrete flow starting from $E_0$ and let $P_\infty$ be the limit of the non-increasing sequence $P(E_h^n)$. Then either one of the following holds: 
	\begin{itemize}
		\item[i)] $(E_h^n)_{n \in \N}$ converges to a disjoint union of $l$ discs of equal radii and total area $m$, where $l=\pi^{-1} (4m)^{-1}P_{\infty}^2 \in \N$;
		\item[ii)] $\left((E_h^n)^c\right)_{n \in \N}$ converges to a disjoint union of $l$ discs of equal radii and total area $1-m$, where $l=\pi^{-1} (4-4m)^{-1}P_{\infty}^2 \in \N$;
        \item[iii)]\rosso{
		$(E_h^n)_{n \in \N}$ converges to a disjoint union of $l$ lamellae of total area $m$, with the same slope and $l\le P_{\infty}/2$.
        Moreover, the equality $l=P_{\infty}/2 \in \N$ holds if and only if the limit is given by vertical or horizontal lamellae.}
	\end{itemize}
	In all cases the convergence is exponentially fast in $C^k$ for every $k \in N$.
	\label{teorema convergenza dimensione 2}
\end{teorema}


\subsection{Comments about the proof of the main results}

The first step towards proving our main result Theorem \ref{teorema convergenza esponenziale} is Proposition \ref{teorema convergenza a meno di traslazioni}. More precisely, we prove the convergence up to translations of any discrete flow, starting Hausdorff-close to a strictly stable critical set $E$, to the latter set.  Such a convergence holds in the $C^k-$norm for every $k\in\N$. Since at this point we can not rule out that different subsequences of the discrete flow may converge to different translates of $E$, the subsequent step consists in proving the convergence of the whole flow to a unique translate of the set $E$ (with exponential rate).

In order to prove Proposition \ref{teorema convergenza a meno di traslazioni} we first show (see Step 1 of the proof of the aforementioned proposition) that every long-time limit of the flow is a critical point of the perimeter. When the ambient space is $\R^N$, this implies that the limit points can only be balls or finite union of balls with the same radii. However, in the periodic setting, we may end up with different critical points of the perimeter. Indeed, already in the three dimensional torus $\T^3$ we find a wealth of different critical points in addition to balls: for example,  lamellae, cylinders and gyroids (see Figure~\ref{figure stable sets}) .
\begin{figure}
	\centering
	\includegraphics[scale=0.2]{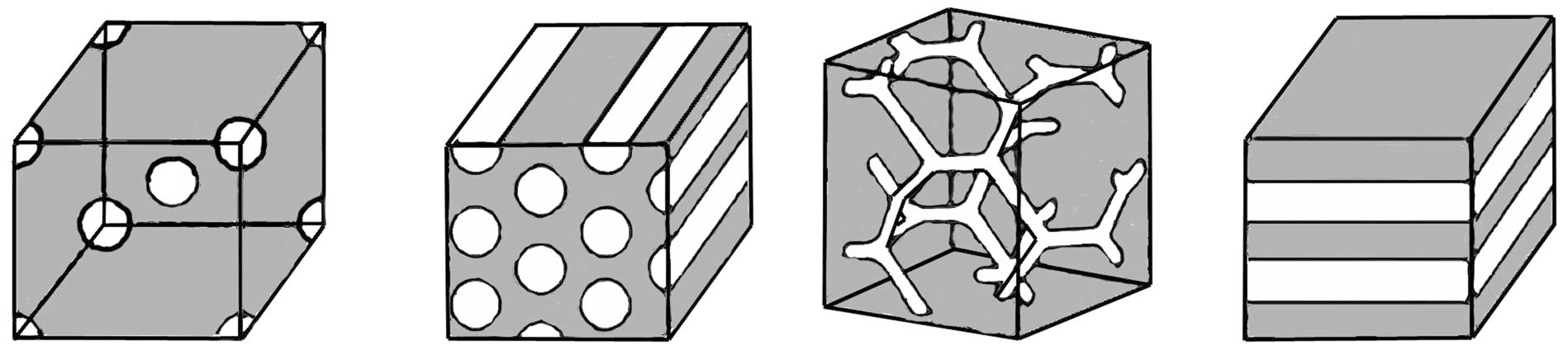}
	\caption{The critical points in $\T^3$. Balls, cylinders, gyroids and $3-$dimensional lamellae.}
	\label{figure stable sets}
\end{figure}
We then exploit the strict stability of $E$ (Proposition \ref{uniform L1 estimate}) to ensure that the flow remains $L^1$-close up to translations to the set $E$. To conclude, a regularity argument shows that the convergence in $L^1$ of the flow to a regular stable set implies the convergence in $C^k$ for every $k\in\N$, thus proving Proposition \ref{teorema convergenza a meno di traslazioni}.

The proof of Proposition \ref{uniform L1 estimate} is based on the following idea: from a stability result in \cite{AFM}, one can estimate the $L^1-$distance (up to translations) of a set $F$ from a strictly stable critical set $E$ in terms of the differences of the perimeters, provided that the $L^1-$distance between $E$ and $F$ remains below a certain threshold. Moreover, a counterexample shows that the Hausdorff-closeness assumption can not be weakened to $L^1-$closeness, as  we will discuss in details in Subsection \ref{sec_contr}.

In order to establish the uniqueness of the limit and, therefore, the main Theorem~\ref{teorema convergenza esponenziale}, Section~\ref{sec convergence} is devoted to proving the convergence of the barycenters of the evolving sets. A crucial intermediate  result consists in generalizing the Alexandrov-type estimate \cite[Theorem~1.3]{MPS} (see also \cite{KM}) to the flat torus. This result provides a stability inequality for $C^1-$normal deformations of strictly stable critical sets in the periodic setting. It could also be seen as an higher-order  \L ojasiewicz-Simon inequality for the perimeter functional.
We briefly give some definitions to present some further details.  Given a set $E$ of class $C^1$ and a function $f:\bd E\to \R$ such that $\|f\|_{L^{\infty}(\bd E)}$ is sufficiently small, the\textit{ normal deformation} $E_f$ of the set $E$ is defined as
\[ \bd E_f:=\{ x+f(x) \nu_E(x)\ :\ x\in\bd E  \}, \]
where $\nu_E$ is the normal outer vector of $E$. A normal deformation $E_f$ is said to be of class $C^k$ if \rosso{$E$ is of class $C^k$ and} $f\in C^k(\bd E)$. The result proved in \cite{MPS} is the following.

\begin{teorema*}
	There exist $\delta \in (0,1/2)$ and $C>0$ with the following property: for any $f \in C^1(\bd B) \cap H^2(\bd B)$ such that $\|f\|_{C^1(\bd B)} \le \delta$, $|E_f|=\omega_N$ and $\textnormal{bar}(E_f)=0$, we have
	\[\|f\|_{H^1(\bd B)} \le C\|H_{E_f}-\overline H_{E_f}\|_{L^2(\bd B)} .\]
\end{teorema*}
We are able to show that in the periodic setting the above quantitative estimate holds with $B$ replaced by any strictly stable critical set. More precisely, we have the following:
\begin{teorema}
	\label{teo alex}
	Let $E \subset \T^N$ be a strictly stable critical set.
	There exist $\delta \in(0,1/2)$ and $C>0$ with the following property: for any $f\in C^1(\bd E)\cap H^2(\bd E)$ such that $\|f\|_{C^1(\bd E)}\leq \delta$ and satisfying
	\begin{equation}
		\left |\int_{\bd E} f\ud \mathcal{H}^{N-1}\right |\le \delta \|f\|_{L^2(\bd E)},\qquad \left |\int_{\bd E} f\nu_E\ud \mathcal H^{N-1}\right|\leq \delta \|f\|_{L^2(\bd E)},
		\label{ipotesi alex}
	\end{equation}
	we have 
	\begin{equation}
		\|f\|_{H^1(\bd E)}\leq C\|H_{E_f}-\overline H_{E_f}\|_{L^2(\bd E)}.
		\label{e:aleq_gen}
	\end{equation}
\end{teorema}
We will prove in details in Section \ref{section Alexandrof} that the conditions \eqref{ipotesi alex} have a geometric explanation. Indeed, the first one ensures that $|E_f|\approx|E|$, up to higher-order error terms, and the second one, \rosso{for some choices of $E$, is implied by imposing $\text{bar}(E_f) \approx\text{bar}(E)$.} 
We finally remark that the estimate \eqref{e:aleq_gen} is optimal for what concerns the power of the norms, see \cite[Remark 1.5]{MPS}.

The last section of the paper is devoted to the two-dimensional case. This particular choice of the dimension is purely technical and it is motivated by the availability of a complete characterization of the critical points of the perimeter in the two-dimensional flat torus. In this setting we are able to prove the exponential convergence of the flow starting from any initial set to either a finite union of balls or a finite union of lamellae or the complement of these configurations.

\subsection*{Acknowledgements}
The authors wish to sincerely thank Professor  Massimiliano Morini for the support provided during the preparation of the paper and for the helpful discussions. We also wish to thank the anonymous referee for the comments which helped to improve the manuscript.
D. De Gennaro has received funding from the European Union’s Horizon 2020 research and innovation program under the Marie Skłodowska-Curie grant agreement No 94532. \includegraphics[scale=0.01]{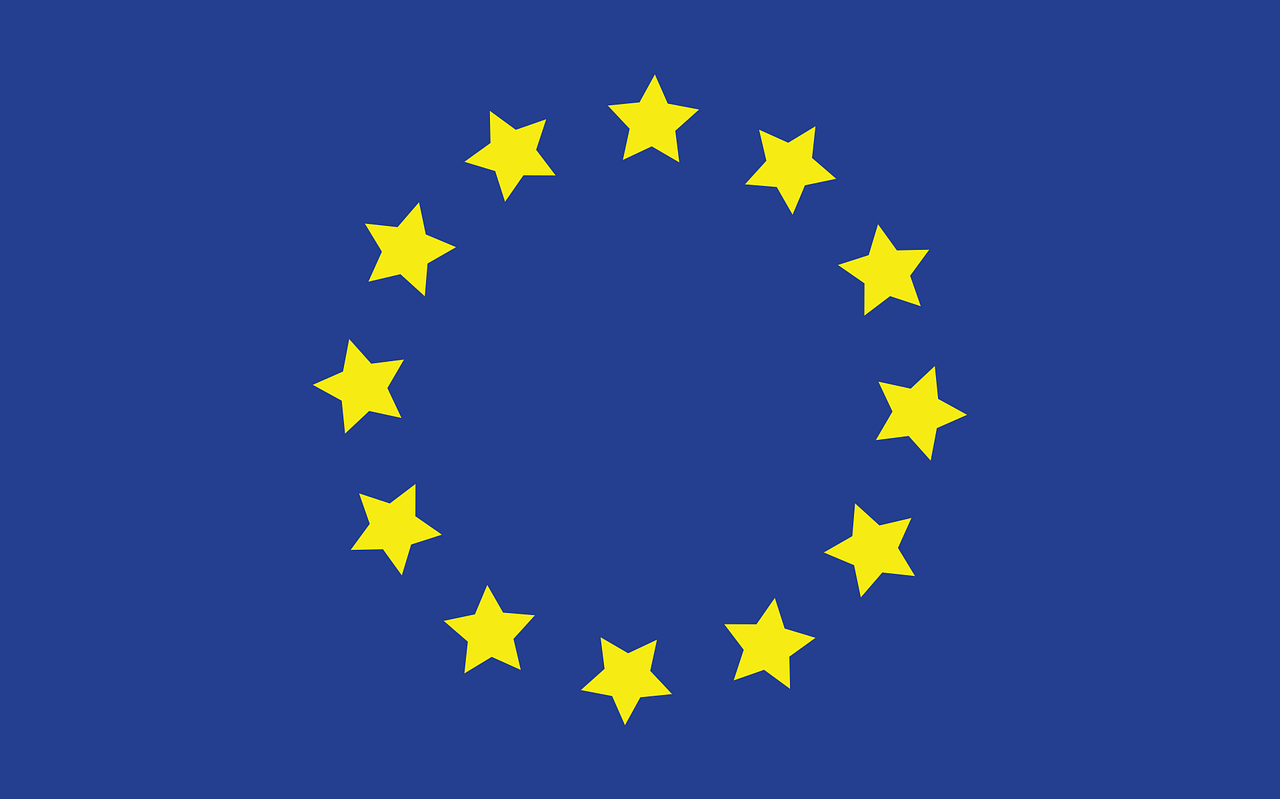}


\section{Preliminaries}
\label{preliminaries}

Let $ \T^N:=\R^N/\mathbb Z^N $ be the $N-$dimensional torus, that is the quotient space $\R^N/ \sim$ where $\sim$ is the equivalence relation given by $ x\sim y $ if and only if $ x-y\in \mathbb Z^N$. We can define the distance between two points $x,y\in\T^N$ simply by 
\[ \dist(x,y)=\min_{z\in\Z^N}|(x+z)-y|. \]
The definition of functional spaces on the torus is straightforward: for example, $W^{k,p}(\T^N)$ can be identified with the subspace of $W^{k,p}_{loc}(\R^N)$ of functions that are one-periodic with respect to all coordinate directions.  When we need to be specific about functions on the torus, it is often convenient to give coordinates to $\T^N$ via the unit cube  $Q=[0,1)^N$.

Firstly, we recall the definition of functions of bounded variation in the periodic setting. We say that a function $u\in L^1(\T^N)$ is of bounded variation if its total variation is finite, that is
\[ |Du|=\sup\left\lbrace \int_{\T^N} u\,\div \varphi \ud x : \varphi\in \rosso{C^1}(\T^N; \R^N),\ |\varphi|\le 1 \right\rbrace <+\infty.\]
We denote the space of such functions by $BV(\T^N)$. We say that a measurable set $E\subset\T^N$ is of finite perimeter in $\T^N$ if its characteristic function $\chi_E\in BV(\T^N)$. The perimeter $P(E)$ of $E$ in $\T^N$ is nothing but the total variation $|D\chi_E|(\T^N)$. We refer to Maggi's book \cite{Mag} for a more complete reference about sets of finite perimeter and their properties.

We recall the following notation.
\begin{defin}
	Let $E$ be a set of class $C^1$.
	Given a function $f:\bd E \to \R$ such that $\|f\|_{L^{\infty}(\bd E)}$ is sufficiently small, we set
	\begin{equation}
		\label{normal_deformation}
		\bd E_f:=\{ x+f(x)\nu_E(x) : x\in \bd E \}
	\end{equation}
	and we call $E_f$ the \textit{normal deformation} of $E$ induced by $f$.
\end{defin}
With a slight abuse of notation, we give the following definition. 
\begin{defin}
	Let $E$ be a set of class $C^1$.
	Let $X(\bd E)$ denote a functional space that can either be $L^p(\bd E)$, $W^{k,p}(\bd E)$, $C^{k,\alpha}(\bd E)$ for some $k\in \N$, $p \ge 1$ and $\alpha \in [0,1]$. For any $F=E_f$ with $f\in  X(\bd E)$, we set  
	\[ \dist_X(F,E)=\|f\|_{X(\bd E)}. \]
\end{defin}
We recall the classical definition of $C^{1,\alpha}-$convergence of sets.
\begin{defin}
	\label{definizione convergenza C^1}
	Given $\alpha\in [0,1]$, a sequence $(E_n)_{n\in\N}$ of $C^{1,\alpha}-$regular sets is said to converge in $C^{1,\alpha}$ to a set $E$ if:
	\begin{itemize}
		\item for any $x\in \bd E$, up to rotations and relabelling the coordinates, we can find a cylinder $C=B'\times (-1,1)$, where $B'\subset \R^{N-1}$ is the unit ball centred at the origin, and functions $f,f_n\in C^{1,\beta}(B';(-1,1))$ such that for $n$ large enough, it holds
		\begin{align*}
			(E-x)\cap C&=\{ (x',x_N)\in B'\times (-1,1)\ :\ x_N\le f (x')  \}\\
			(E_n-x)\cap C&=\{ (x',x_N)\in B'\times (-1,1)\ :\ x_N\le f_n (x')  \};    \end{align*}
		\item it holds 
		\[ f_n\to f\quad \text{in } C^{1,\alpha}(B').  \]
	\end{itemize}
\end{defin}
The following is a simple rephrasing of a classical result concerning the $C^{1,\alpha}-$convergence of $\Lambda-$minimizers of the perimeter (see e.g. \cite[Theorem 4.2]{AFM}).

\begin{teorema}
	\label{teorema convergenza C^1}
	Let $\Lambda >0$ and let $E$ be a set of class $C^2$. Then for every $\varepsilon>0$, there exists $\delta=\delta(\varepsilon,E)>0$ with the following property: for every $\Lambda-$minimizer $F$ such that $|E\triangle F|\le \delta$, then $F$ is of class $C^{1,1/2}$ and
	\[\dist_{C^{1,\beta}}(E,F) \le \varepsilon\quad \text{for} \quad \beta\in (0,1/2). \]
\end{teorema}

We now recall some preliminary results from \cite{AFM} regarding the second variation of the perimeter in the flat torus. 
Firstly, we fix some notation. Let $E\subset \T^N$ be a set of class $C^2$ and let $\nu_E$ be its exterior normal.
Throughout the section, when no confusion is possible, we shall omit the subscript $E$ and write $\nu$ instead of $\nu_E$. Given a vector $X$, its tangential part on $\bd E$ is defined as $X_\tau=X-(X\cdot \nu)\nu$. In particular, we will denote by $D_\tau$ the tangential gradient operator given by $D_\tau \varphi=(D\varphi)_\tau$ . We also recall that the second fundamental form $B_E$ of $\bd E$ is given by $D_\tau \nu$, its eigenvalues are called principal curvatures and its trace is called mean curvature, and we denote it by $H_E$.

Let $X:\T^N\to\rosso{\R^N}$ be a vector field of class $C^2$. Consider the associated flow $\Phi:\T^N\times (-1,1)\to\T^N$ defined by $\frac {\bd \Phi}{\bd t}=X(\Phi),\ \Phi(\cdot,0)=Id$. We define the \textit{first and second variation of the perimeter at $E$} with respect to the flow $\Phi$ to be respectively the values
\[\dfrac {\ud} {\ud t}\Big\lvert_{t=0}P(E_t), \quad \dfrac {\ud^2} {\ud t^2}\Big\lvert_{t=0}P(E_t) \]
where $E_t=\Phi(\cdot,t)(E).$ It is a classical result of the theory of sets of finite perimeter (see \cite{Mag}) that the the first variation of the perimeter has the following expression
\[ \dfrac {\ud} {\ud t}\Big\lvert_{t=0}P(E_t)=\int_{\bd^* E} H_E \nu_E\cdot X\udH,   \]
where $H_E$ is the (weak) scalar curvature of $E$. The following equation for the second variation of the perimeter holds.
\begin{teorema}[Theorem 3.1 in \cite{AFM}]
	If $E$, $X$ and $\nu$ are as above, we have
	\begin{align*}
		\dfrac {\ud^2} {\ud t^2}\Big\lvert_{t=0}P(E_t)=&\int_{\bd E} \left( |D_\tau (X\cdot \nu)|^2-|B_E|^2 (X\cdot \nu)^2 \right)\udH -\int_{\bd E} H_E \div_\tau (X_\tau(X\cdot \nu)) \udH\\
		&+\int_{\bd E} H_E(\div X)(X\cdot \nu)  \udH.
	\end{align*}
\end{teorema}

\begin{oss}
	We remark that the last two integral in the above expression vanish when $E$ is a critical set for the perimeter and if $|\Phi(\cdot,t)(E)|=|E|$ for all $t\in  [0,1]$. Indeed, if $E$ is a regular critical set for the perimeter then its curvature is constant, therefore the second integral vanishes. Moreover, if the flow $\Phi$ is volume-preserving then it can be shown (see equation (2.30) in \cite{CS}) that 
	\[ 0=\dfrac{\ud^2 |E_t|}{\ud t^2}=\int_{\bd E} (\div X)(X\cdot \nu)\udH. \]
	Hence, if $\Phi$ is a volume-preserving variation of a regular critical set $E$ we have
	\[ \dfrac {\ud^2} {\ud t^2}\Big\lvert_{t=0} P(E_t)=\int_{\bd E} \left( |D_\tau (X\cdot \nu)|^2 -|B_E|^2 (X\cdot \nu)^2  \right)\udH=: \delta ^2 P(E)[X\cdot \nu_E].\]
\end{oss}
We remark that due to the translation invariance of the perimeter functional, the second variation degenerates along flows of the form $\Phi(x,t)=x+t\eta,$ where 
$\eta\in \R^N$. In view of this it is convenient to introduce the subspace $T(\bd E)$ of $\tilde H^1(\bd E):=\left\lbrace \varphi\in H^1(\bd E) : \int_{\bd E} \varphi \udH=0   \right\rbrace$ generated by the functions $\nu_i,$ $i=1,\dots, N$. Its orthogonal subspace, in the $L^2-$sense, will be denoted by $T^\perp(\bd E)$ and is given by
\[ T^\perp(\bd E)= \left\lbrace \varphi\in\tilde H^1(\bd E) : \int_{\bd E} \varphi  \nu_i\udH=0,\ i=1,\dots,N \right\rbrace.\]

\begin{defin}
	\label{definizione strictly stable set}
	We say that a regular critical set $E$ is \textit{a strictly stable set} if it has positive second variation of the perimeter, in the sense that 
	\[\delta^2P(E)[\varphi]>0,\qquad \forall \varphi \in T^{\perp} (\bd E) \setminus \{0\}.\]
\end{defin}

The following result ensures that the second variation of a strictly stable set $E$ is coercive on the subspace $T^\perp(\bd E).$
\begin{lemma}[Lemma 3.6 in \cite{AFM}]
	\label{Lemma_3.6_AFM}
	Assume that $E$ is a strictly stable set, then
	\[m_0:= \inf\{ \delta^2 P(E)[\varphi]: \varphi \in T^{\perp}(\bd E),\, \| \varphi \|_{H^1(\bd E)}=1\}>0 \]
	and
	\[\delta^2 P(E)[\varphi] \geq m_0 \|\varphi \|^2_{H^1(\bd E)} \quad \forall \varphi \in T^{\perp}(\bd E). \]
\end{lemma}
Moreover, from the Step 1 in the proof of \cite[Theorem 3.9]{AFM} we obtain also the following result.
\begin{lemma}
	Assume that $E$ is a strictly stable set, then
	\[ \inf\left\{ \delta^2 P(E)[\varphi]: \varphi\in \tilde H^1,\  \| \varphi \|_{H^1(\bd E)}=1,\ \left| \int_{\bd E}\varphi\nu_E\udH  \right|\leq \delta \right\}\geq \dfrac{m_0}2,\]
	where the constant $m_0$ is the one in Lemma \ref{Lemma_3.6_AFM}.
	\label{Lemma_3.6_AFM_migl}
\end{lemma}

In the proof of our main result we will also need the following key lemma which shows
that any set $F$ sufficiently close to $E$ can be translated in such a way that the resulting set $\tilde F$ satisfies $\bd \tilde F= \{ x+\varphi(x)\nu_E(x) : x\in\bd E \}$, with $\varphi$ having a suitably small projection on $T(\bd E)$.
\begin{lemma}[Lemma 3.8 in \cite{AFM}]
	Let $E\subset \T^N$ be of class $C^3$ and let $p>N-1$. For every $\delta>0$ there exist $C>0$ and $\eta_0>0$ such that if $F\subset\T^N$ satisfies $\bd F=\{ x+\psi(x)\nu_E(x) : x\in\bd E \}$ for some $\psi\in C^2({\bd E})$ with $\|\psi\|_{W^{2,p}(\bd E)}\le \eta_0$, then there exist $\sigma\in\T^N$ and $\varphi\in W^{2,p}(\bd E)$ with the properties that 
	\[ |\sigma|\le C\|\psi\|_{W^{2,p}(\bd E)},\quad \|\varphi\|_{W^{2,p}(\bd E)}\le C\|\psi\|_{W^{2,p}(\bd E)} \]
	and
	\[\bd F+\sigma=\{ x+\varphi(x)\nu_E(x) : x\in\bd E \},\quad \left\lvert \int_{\bd E}\varphi\nu_E\udH  \right\rvert\le \delta \|\varphi\|_{L^2(\bd E)}.\]
	\label{lemma 3.8_AFM}
\end{lemma}

Let $E, F\subset \T^N$ be measurable sets. We define
\begin{equation*}
	\alpha(E,F) := \min_{x \in \T^N} |E \triangle (F+x)|.
\end{equation*}
In one of the main results of \cite{AFM} the authors proved that the distance $\alpha(\cdot,\cdot)$ between a set and a strictly stable set can be bounded by the square root of the difference of their perimeters.
\begin{teorema}[Corollary 1.2 in \cite{AFM}]
	\label{coroll 1.2}
	Let $E\subset\T^N$ be a strictly stable set. Then, there exist $ \sigma=\sigma(E)$, $C=C(E)>0$ such that 
	\[ C\alpha^2(E,F)\leq P(F)-P(E) \]
	for all $F\subset \T^N$ with $|F|=|E|$ and $\alpha(E,F)< \sigma$.
\end{teorema}

\section{A quantitative generalized Alexandrov Theorem}
\label{section Alexandrof}

In this section, we will prove that local minimizers of the perimeter in the flat torus satisfy a quantitative Alexandrov-type estimate. We reproduce some arguments similar to the ones used in the proof of Theorem~1.3 in \cite{MPS}.
In this section, we consider $E\subset \T^N$ a strictly stable set. Thanks to some classical results for sets of finite perimeter (see for example \cite[Theorem  27.4]{Mag}), the previous hypothesis implies that $E$ is connected and it is of class~$C^{\infty}$.

First of all, we compute the $(N-1)-$Jacobian of the map \[\Phi:\bd E\to\bd E_f \subset \R^N, \qquad x \mapsto x+f(x)\nu_E(x).\]
Given $x \in \bd E$, we choose an orthonormal basis
\[\mathcal{B'}:=\left\lbrace v_1(x),\dots ,v_{N-1}(x)\right\rbrace\]
of $T_x  E$ such that in this basis the second fundamental form of $E$, $B_E(x): T_x  E\to T_{x} E \subset \R^N$,
has the following expression
\[B_E(x)=\begin{pmatrix} 
	\kappa_1(x)  &  & \\
	&\ddots  & \\
	&&\kappa_{N-1}(x) \\
	0&\dots & 0
\end{pmatrix}, \]
where $\kappa_1(x), \ldots, \kappa_{N-1}(x)$ are the principal curvatures of $E$ in $x$. 
We then complete $\mathcal{B'}$ to a basis $\mathcal{B}$ of the whole $\R^N$ with the normal vector $v_N(x):=\nu_E(x)$. 
In the following, to simplify the notation, we will drop the dependence on $x$. 
The tangential differential 
of $\Phi$ with respect to the basis $\mathcal{B}$ is given by
\[ D\Phi=I+\nu_E\otimes \nabla f+f D\nu_E,\]
where $I$ is the immersion $T_x  E\hookrightarrow \R^N$, $\nabla f$ is the tangential gradient of $f$ and $D\nu_E$ is the tangential differential of $\nu_E$. 
Given the regularity of $\bd E$, we recall that $D\nu_E$ is equal to $B_E$. 
Moreover, by definition of $\mathcal{B}$, we have that 
\begin{equation*}
	(\nu_E \otimes \nabla f)(v_i, v_j)= \delta_{N,i}\, \nabla f \cdot v_j, \quad i=1, \ldots,N, \,\, j =1, \ldots, N-1.
\end{equation*}
Thanks to the previous observations we obtain
\begin{equation*}
	D\Phi=
	\begin{pmatrix} 
		1 &  &  \\
		&\ddots  & \\
		&&1\\
		0 &\dots & 0
	\end{pmatrix} 
	+
	\begin{pmatrix} 
		0 &  & \\
		&\ddots  & \\
		&&0\\
		\bd_{v_1} f &\dots & \bd_{v_{N-1}} f
	\end{pmatrix} 
	+
	\begin{pmatrix} 
		\kappa_1 f &  & \\
		&\ddots  & \\
		&&\kappa_{N-1} f\\
		0 &\dots & 0
	\end{pmatrix} ,
\end{equation*}
thus we find the following expression
\begin{equation}
	D\Phi=\begin{pmatrix} 
		1+\kappa_1 f &  & \\
		&\ddots  & \\
		&&1+\kappa_{N-1} f\\
		\bd_{v_1} f &\dots & \bd_{v_{N-1}} f
	\end{pmatrix}.
	\label{D Phi}
\end{equation}
By Binet formula, the Jacobian $J\Phi$ can be explicitly computed as
\begin{align}
	J \Phi&=\left (\prod_{i=1}^{N-1}(1+\kappa_i f)^2+\sum_{j=1}^{N-1} (\bd_{v_j} f)^2\prod_{i\neq j}(1+\kappa_i f)^2\right )^{1/2}\nonumber\\
	&=\prod_{i=1}^{N-1}(1+\kappa_i f)\left(  1+ \sum_{j=1}^{N-1} \dfrac{(\bd_{v_j} f)^2}{(1+\kappa_j f)^2} \right)^{1/2}.
	\label{Jac_gen}
\end{align}
To show the previous formula, we characterize the minors of $D\Phi$. If we omit the $N-$th row of $D\Phi$, we obtain the minor
\[ M_N=\begin{pmatrix} 
	1+\kappa_1 f &  & \\
	&\ddots  & \\
	&&1+\kappa_{N-1} f
\end{pmatrix}, \]
if we omit the $i-$th row of $D \Phi$ for $1\le i\le N-1$, we obtain the minor 
\[M_i=\begin{pmatrix} 
	1+\kappa_1 f &  && & &  & \\
	& \ddots& &  & &&& \\
	& & &1+\kappa_{i-1} f & & && \\
	& & & &1+\kappa_{i+1} f & && \\
	& & && & & \ddots && \\
	& & && & & & 1+\kappa_{N-1}f\\    
	\bd_{v_1} f & \dots&&\bd_{v_{i-1}}f & \bd_{v_{i+1}}f & &\dots & \bd_{v_{N-1}} f
\end{pmatrix}.\]
We then deduce \eqref{Jac_gen} by explicitly computing
\[ \det(M_N)^2=\prod_{i=1}^{N-1}(1+\kappa_i f)^2,\quad \det (M_i)^2=(\bd_{v_i} f)^2\prod_{j\neq i}(1+\kappa_j f)^2 .\]

The previous formula for $J\Phi$ allows us to calculate some quantities that will be useful later on. Observe that, if $\|f\|_{C^1}$ is small enough, the map $\Phi$ is a diffeomorphism from $\bd E$ to $\Phi(\bd E)=\bd E_f$, and thus the tangential differential $D\Phi: T_x  E\to T_{\Phi(x)} E_f$ is a surjective map. In particular, this allows us to calculate the normal \rosso{vector} $\nu_{E_f}$ in $\Phi(x)$. We remark that a vector $v$ orthogonal to every column of \eqref{D Phi} is a normal vector to the whole tangent space $T_{\Phi(x)} E_f$, therefore a possible $v$ is given by
\[ v=-\sum_{i=1}^{N-1} \frac{\bd_{v_i} f}{1+\kappa_i f}v_i+\nu_E, \]
where the sign of the component along $\nu_E$ is taken positive so that the case $f=0$ is consistent with the orientation of $\nu_E$.
Since $|v|\ge 1$, by normalizing $v$ we obtain the normal \rosso{vector} 
\begin{equation}
	\nu_{E_f}=\left(\nu_E-\sum_{i=1}^{N-1} \frac{\bd_{v_i} f}{1+\kappa_i f}v_i\right)  \left(  1+ \sum_{j=1}^{N-1} \dfrac{(\bd_{v_j} f)^2}{(1+\kappa_j f)^2} \right)^{-1/2},
	\label{vers normale E_f}
\end{equation}
moreover, we remark that
\begin{equation}
	\nu_E\cdot \nu_{E_f}=\left(  1+ \sum_{j=1}^{N-1} \dfrac{(\bd_{v_j} f)^2}{(1+\kappa_j f)^2} \right)^{-1/2}.
	\label{prod scal vers normali}
\end{equation}

We can now compute explicitly the formula for the first variation of the perimeter.
\begin{lemma}
	Setting $Q:=\left(  1+ \sum_{j=1}^{N-1} \dfrac{(\bd_{v_j} f)^2}{(1+\kappa_j f)^2} \right)^{1/2}$, the following formulas hold true:
	\begin{enumerate}
		\item If $f\in L^\infty(\bd E)\cap H^1(\bd E)$ with $\|f\|_{L^\infty}$ sufficiently small, then
		\[ P(E_f)=\int_{\bd E}Q\prod_{i=1}^{N-1}(1+\kappa_i f) \ud \mathcal H^{N-1}. \]
		\item If $f\in L^{\infty}(\bd E)\cap H^1(\bd E)$ with $\|f\|_{L^\infty} $ sufficiently small, then the first variation $\delta P(E_f)[\varphi]$ exists for all $\varphi\in C^1(\bd E)$ and is given by
		\begin{align}
			\delta P(E_f)[\varphi]
			=&\int_{\bd E} \varphi \,Q\sum_{i=1}^{N-1}\kappa_i\prod_{j\neq i}(1+\kappa_j f)  \udH \nonumber\\
			&+\int_{\bd E}\dfrac 1Q\prod_{i=1}^{N-1}(1+\kappa_i f) \left (\sum_{j=1}^{N-1}\frac{\bd_{v_j} \varphi \,  \bd_{v_j} f}{(1+k_j f)^2} -
			\varphi\,\sum_{j=1}^{N-1}\frac{ k_j\,(\bd_{v_j} f)^2}{(1+k_j f)^3}\right)\udH.
			\label{variaz_gen}
		\end{align}
	\end{enumerate}
\end{lemma}
\begin{proof}
	The first formula is a straightforward consequence of the area formula
	\[ P(E_f)=\int_{\bd E_f}\ud \mathcal H^{N-1}=\int_{\bd E}J\Phi \ud \mathcal H^{N-1} \]
	and of the expression of the Jacobian $J\Phi$ in \eqref{Jac_gen}. Now, \eqref{variaz_gen} easily follows by taking the derivatives
	\[ \dfrac{d}{d\varepsilon}\Big|_{\varepsilon=0}  \ P(E_{f+\varepsilon \varphi})\]
	in the first formula.
\end{proof}

In the following, with $C$ we will refer to a positive constant, possibly changing from line to line, and we will specify its explicit dependence when needed.
\begin{oss}
	\label{oss}
	We observe that, if $\|f\|_{L^{\infty}(\bd E)}$ is small enough and $|E_f|=|E|$, then there exists a constant $ C>0$, only depending on $E$, such that
	\begin{equation}
		\left |\int_{\bd E}f(x)\udH(x)\right|\leq  C\int_{\bd E} f(x)^2\udH(x).
		\label{conto media curvatura non nulla}
	\end{equation}  
	Firstly, since $\bd E$ is regular, for every $\varepsilon >0$ sufficiently small there exists a tubular neighborhood $\mathcal{N}$ of $\bd E$ such that $\mathcal{N}$ is diffeomorphic to $\bd E\times (-\varepsilon,\varepsilon)$ via the diffeomeorphism $\Psi(x,t)=x+\nu_E(x)t$.
	The Jacobian of $\Psi$ is given by
	\begin{equation}\label{quiquo}
		J\Psi(x,t)=\prod_{i=1}^{N-1}(1+\kappa_i(x) t). 
	\end{equation} 
	Secondly, if $\|f\|_{L^{\infty}(\bd E)}$ is small enough, we remark that the condition $|E_f|=|E|$ is equivalent to
	\[ 0=|E_f|-|E|=\int_{\bd E}\int_0^{f(x)}J\Psi(x,t)\ud t\ud \mathcal H^{N-1}(x). \] 
	Then, we can conclude that
	\begin{align*}
		0&=\int_{\bd E}\int_0^{f(x)}J\Psi(x,t)\ud t\ud \mathcal H^{N-1}(x)\\
		&=\int_{\bd E}f(x)\udH(x)+\int_{\bd E}\int_0^{f(x)}(J\Psi(x,t)-1) \ud t\ud \mathcal H^{N-1}(x)  \\
		&= \int_{\bd E}f(x)\udH(x)+\int_{\bd E}\int_0^{f(x)} (H_E(x) \, t + o(t))\ud t \udH(x),
	\end{align*}
	that implies \eqref{conto media curvatura non nulla} for a constant depending only on $N$ and the principal curvatures of $E$.
\end{oss}

We are now able to prove the following stability result; it ensures that the second variation of the perimeter remains strictly positive for small normal deformations of a strictly stable set~$E$.
\begin{lemma}
	Fix $N\geq 2$.
	There exists $\delta=\delta(E)>0$ small such that, if $f\in L^{\infty}(\bd E)\cap H^1(\bd E)$ with $\|f\|_{L^\infty(\bd E)}\le \delta$,
	\begin{equation}\label{1988}
		\left\lvert\int_{\bd E} f(x)\ud \mathcal H^{N-1}(x)\right\rvert\leq \delta \|f\|_{L^2(\bd E)} \quad \text{and} \quad
		\left \lvert\int_{\bd E} f(x)\nu_E(x)\ud \mathcal H^{N-1}(x)\right \rvert\leq \delta \|f\|_{L^2(\bd E)},
	\end{equation}
	then we have
	\begin{equation*}
		\delta^2 P(E)[f]=\int_{\bd E}(|\nabla f(x)|^2-|B_E(x)|^2 f(x)^2)\udH(x)\geq \frac{m_0}{8} \|f\|_{H^1(\bd E)}^2,
	\end{equation*}
	where $m_0$ is the constant given by Lemma \ref{Lemma_3.6_AFM}.
	\label{lemma 1.4}
\end{lemma}

\begin{proof}
	Set $g=f-\bar f$, where $\bar f=\fint_{\bd E}f \ud \mathcal{H}^{N-1}$, then $g$ has zero average and, by the first inequality in \eqref{1988}, we have
	\begin{equation}\label{nuova1980}
		\bar f^2= \frac{1}{P(E)^2} \left (\int_{\bd E} f \ud \mathcal{H}^{N-1}\right )^2  \le 
		C \delta^2\|f\|^2_{L^2(\bd E)}.
	\end{equation}
	If $\delta$ is sufficiently small, from \eqref{nuova1980} we obtain
	\[\|g\|_{L^2(\bd E)}^2 =\|f-\bar f\|_{L^2(\bd E)}^2= \|f\|_{L^2(\bd E)}^2-\bar f^2 P(E) \geq \|f\|_{L^2(\bd E)}^2\left(1- C\delta^2   \right ) \geq \frac{1}{2}\|f\|_{L^2(\bd E)}^2. \]
	Using the previous inequality, \eqref{nuova1980} again and the second inequality in \eqref{1988} we infer that the function $g$ satisfies
	\[\left\lvert\int_{\bd E} g\nu_E\ud \mathcal H^{N-1}  \right\rvert\leq \left\lvert\int_{\bd E} f\nu_E\ud \mathcal H^{N-1}  \right\rvert+\left\lvert\int_{\bd E} \bar f\nu_E\ud \mathcal H^{N-1}  \right\rvert \leq C \delta\|g\|_{L^2(\bd E)}.\]
	Then, we can apply Lemma \ref{Lemma_3.6_AFM_migl} to obtain
	\[\delta^2 P(E)[g]\geq\dfrac{m_0}2\|g\|_{H^1(\bd E)}^2,\]
	provided $\delta$ small enough. We conclude
	\begin{align*}
		\delta^2P(E)[f]&= \delta^2P(E)[g]-\delta^2 P(E)[g ]+\delta^2 P(E)[f]\\
		&=\delta^2P(E)[g] -2\bar f \int_{\bd E} |B_E(x)|^2f(x) \ud \mathcal{H}^{N-1}(x)+\bar f^2 \int_{\bd E} |B_E(x)|^2 \ud \mathcal{H}^{N-1}(x) \\
		&\geq \dfrac{m_0}2\|g\|_{H^1(\bd E)}^2-C|\bar f | \|f\|_{L^2(\bd E)}
		\ge \dfrac{m_0}2(\|g\|_{L^2(\bd E)}^2+\|\nabla g\|_{L^2(\bd E)}^2)-C\delta \|f\|_{L^2(\bd E)}^2\\
		&\ge \dfrac{m_0}4(\|f\|_{L^2(\bd E)}^2+\|\nabla f\|_{L^2(\bd E)}^2)-C\delta \|f\|_{L^2(\bd E)}^2
		\geq\dfrac{m_0}8 \|f\|_{H^1(\bd E)}^2,
	\end{align*}
	up to taking $\delta$ smaller if needed, and where the constant $C>0$ only depends on $E$.
\end{proof}

\begin{oss}
	Remark \ref{oss} ensures that the conclusion of the previous lemma also holds if we replace the hypothesis $|\int_{\bd E} f\ud \mathcal H^{N-1}|\le \delta \|f\|_{L^2(\bd E)}$ with $\|f\|_{L^\infty(\bd E)}$ small enough and $|E_f|=|E|$.
\end{oss}

We are now able to prove the generalized version of the quantitative Alexandrov's inequality in the periodic setting, Theorem \ref{teo alex}.

\begin{proof}[Proof of Theorem \ref{teo alex}]
	First of all we notice that, if we take the constant $C$ in \eqref{e:aleq_gen} to be bigger than $\sqrt{P(E)/2}$, then it is enough to consider only the case $\|H_{E_f}-\overline H_{E_f}\|_{L^2(\bd E)}\leq 1$.

	Set $p=x+f(x)\nu_E(x)$ and let $\varphi \in C^1(\bd E)$, by the definition of scalar mean curvature $H_{E_f}$ and a change of coordinates we obtain
	\begin{equation}
		\delta P(E_f)[\varphi]=\int_{\bd E} (H_{E_f}\nu_{E_f})(p)\cdot \nu_E\,\varphi\, J\Phi\ud \mathcal H^{N-1}.
		\label{def_variaz_gen}
	\end{equation}
	Combining \eqref{def_variaz_gen}, \eqref{Jac_gen} and \eqref{prod scal vers normali} we obtain
	\[\delta P(E_f)[\varphi]=\int_{\bd E}H_{E_f}\varphi\,J\Phi\left(  1+ \sum_{j=1}^{N-1} \dfrac{(\bd_{v_j} f)^2}{(1+\kappa_j f)^2} \right)^{-1/2} \ud \mathcal H^{N-1}=\int_{\bd E}H_{E_f}\varphi \prod_{i=1}^{N-1}(1+\kappa_i f)\udH.\]
	In the following, with a slight abuse of notation, with the symbol $O(g)$ we will mean any function $h$ of the form $h(x)=r(x)g(x)$, where $|r(x)|\le C$ for all $x\in \bd E$ and $C$ is a constant depending only on $N$ and $E$.

	By a simple Taylor expansion we have
	\begin{equation}
		\delta P(E_f)[\varphi]=\int_{\bd E}H_{E_f}\varphi\,\left(1+H_E f+O(f^2)  \right) \ud \mathcal H^{N-1}. 
		\label{cp_1_1_gen}
	\end{equation}
	From \eqref{variaz_gen} and again by Taylor expansion, we obtain
	\begin{align}
		\delta P(E_f)[\varphi]=&\int_{\bd E} \left (H_E +f\sum_{i=1}^{N-1} \kappa_i \sum_{s\ne i} \kappa_s  + O(f^2)  + O(| \nabla f |^2)\right)\varphi\ud \mathcal H^{N-1}   \nonumber\\
		&+\int_{\bd E}(\nabla f + h) \cdot \nabla \varphi\udH\nonumber\\
		=&  \int_{\bd E} \left( H_E +f H^2_E -|B_E|^2f + O(f^2)  + O(| \nabla f |^2)  \right) \varphi  \ud \mathcal H^{N-1} \nonumber\\
		&+\int_{\bd E}(\nabla f + h) \cdot \nabla \varphi\udH
		\label{cp_2_gen}
	\end{align}
	where $\nabla f,\nabla \varphi$ are respectively the tangent gradient of $f,\varphi$ on $\bd E$  and  $h$ is a vector field satisfying $|h|\leq C(|f|+|\nabla f|^2))|\nabla f|$.
	Set $R=O(f^2)+O(|\nabla f|^2)$, by comparing \eqref{cp_1_1_gen} and \eqref{cp_2_gen} we infer that
	\begin{align}
		\int_{\bd E}(\nabla f\cdot\nabla \varphi-|B_E|^2 f\varphi) \udH\nonumber =&\int_{\bd E}(H_{E_f}-H_E) \left(1+H_E f+ R\right)\varphi\udH \nonumber\\
		&-\int_{\bd E}  (h\cdot\nabla \varphi+(O(f^2)+O(|\nabla f|^2))\varphi) \udH.
		\label{eq_lin3_gen}
	\end{align}
	Testing \eqref{eq_lin3_gen} with $\varphi=1$, we get
	\begin{align*}
		\int_{\bd E} (H_{E_f}-H_E)\left( 1+H_E f + R \right)\udH
		=\int_{\bd E}(O(|f|)+O(|\nabla f|^2))\ud \mathcal H^{N-1},
	\end{align*}
	then, for $\delta$ sufficiently small, using Hölder inequality we obtain 
	\begin{align*}
		\left| \overline H_{E_f}-H_E \right|
		&=\left|-\fint_{\bd E}  (H_{E_f}-H_E)(H_E f+R)\udH + \fint_{\bd E}(O(|f|)+O(|\nabla f|^2))\ud \mathcal H^{N-1}\right|\\
		&\le \left| \fint_{\bd E}  (H_{E_f}-\overline H_{E_f})(H_E f+R)\udH \right| + \left| \fint_{\bd E}  (\overline H_{E_f}-H_E) (H_E f+R)\udH \right| \\
		&\ \ + \int_{\bd E}(O(|f|)  +O(|\nabla f|^2))\ud \mathcal H^{N-1}\\
		&\le \delta\, \dfrac{|H_E|+C\delta}{P(E)}\|H_{E_f}-\overline H_{E_f}\|_{L^2}+\delta  \left(|H_E|+C\delta \right )|\overline H_{E_f}-H_E|\\
		&\ \ + \int_{\bd E}(O(|f|) +O(|\nabla f|^2))\ud \mathcal H^{N-1},
	\end{align*}
	with $C=C(N,E)$ since $\delta\le 1$. 
	For $\delta $ small enough, recalling that $\|H_{E_f} - \overline H_{E_f}\|_{L^2}\leq1$, the previous inequality implies
	\begin{equation}
		\frac 12 |\overline H_{E_f}-H_E|\le  C\delta \|H_{E_f}-\overline H_{E_f}\|_{L^2}+ \int_{\bd E}(O(|f|)  +O(|\nabla f|^2))\ud \mathcal H^{N-1} \le C\delta.
		\label{eq_lin7_gen}
	\end{equation}
	Using the bound $\|f\|_{C^1}\leq \delta$ and the definition of $h$ we easily see that
	\[h\cdot\nabla f=\delta \,O(|\nabla f|^2) . \]
	Testing \eqref{eq_lin3_gen} with $\varphi=f$, using Hölder's inequality and by the previous remark, we get
	\begin{align}
		\label{fine_catena_gen}
		\int_{\bd E} (|\nabla f|^2&-|B_E|^2 f^2)\udH\nonumber = \int_{\bd E}\left(  H_{E_f}-H_E \right)( 1+H_E f + R)f\udH\nonumber\\
		&+\delta\int_{\bd E}(O(f^2) +O(|\nabla f|^2))\ud \mathcal H^{N-1}\nonumber\\
		=& \int (H_{E_f}-\overline H_{E_f})(1+H_E f + R)f\udH\nonumber+\int(\overline H_{E_f}-H_E)(1+H_E f + R)f\udH\nonumber\\
		& +\delta\int_{\bd E}(O(f^2) +O(|\nabla f|^2))\ud \mathcal H^{N-1}\nonumber\\
		\leq&\, C\|H_{E_f}-\overline H_{E_f}\|_{L^2}\|f\|_{L^2}+ \nonumber |\overline H_{E_f}-H_E|\int (1+H_E f + R)f\udH\nonumber\nonumber\\
		&+\delta\int_{\bd E}(O(f^2) +O(|\nabla f|^2))\ud \mathcal H^{N-1}\nonumber\\
		=&\, C\|H_{E_f}-\overline H_{E_f}\|_{L^2}\|f\|_{L^2}+|\overline H_{E_f}-H_E|\int (f+O(f^2)+f O(|\nabla f|^2))\udH\nonumber\\
		&+\delta\int_{\bd E}(O(f^2) +O(|\nabla f|^2))\ud \mathcal H^{N-1}.
	\end{align}
	By \eqref{conto media curvatura non nulla}, \eqref{eq_lin7_gen} and by Hölder inequality, we obtain
	\[|\overline H_{E_f}-H_E|\int (f+O(f^2)+f O(|\nabla f|^2))\udH  \le  \delta \int_{\bd E}(O(f^2)+O(|\nabla f|^2)).\]
	Finally, by the above inequality, \eqref{conto media curvatura non nulla} again and by combining \eqref{fine_catena_gen} with \eqref{eq_lin7_gen} 
	we deduce that, for any $\eta>0$, it holds 
	\begin{align}
		\int_{\bd E}(|\nabla f|^2-|B_E|^2 f^2)\ud \mathcal H^{N-1}&\leq C\|H_{E_f}-\overline H_{E_f}\|_{L^2}\|f\|_{H^1}+\delta \int_{\bd E}(O( f^2)+O(|\nabla f|^2))\ud \mathcal H^{N-1}\nonumber\\
		&\leq \dfrac 1{\eta} C^2\|H_{E_f}-\overline H_{E_f}\|_{L^2}^2+\eta\|f\|_{H^1}^2+C\delta \|f\|_{H^1}^2.
		\label{fine_gen}
	\end{align}
	The conclusion then follows combining \eqref{fine_gen} with Lemma \ref{lemma 1.4} and taking $\delta$ and $\eta$ sufficiently small.
\end{proof}

\begin{oss}
	For some particular choices of the set $E$, a geometric explanation of the condition  
	\begin{equation}
		\left \lvert\int_{\bd E} f\nu_E\ud \mathcal H^{N-1}\right \rvert\leq \delta \|f\|_{L^2}
		\label{condition strange}
	\end{equation}
	can be found. It is the case for the ball, the cylinder or the lamella. For example consider $E=B_r$, the case where $E$ is a cylinder or a lamella being analogous. We show that in this case condition \eqref{condition strange} follows from enforcing 
	\[\text{bar}(E_f)=\text{bar}(B_r)=0.\]
	Indeed, consider the case $r=1$ for simplicity, the barycenter in polar coordinates is given by
	\[0=\dfrac 1{(N+1)\omega_N}\int_{\bd B} (1+f)^{N+1}x\ud\mathcal H^{N+1}\]
	and thus, by a simple Taylor expansion, we obtain
	\begin{align*}
		0&=\int_{\bd B} \left( 1+(N+1)f+\frac 12 R f^2 \right) x \ud \mathcal H^{N-1}\nonumber \\
		&=(N+1)\int_{\bd B} f(x)x   \ud \mathcal H^{N-1} +\frac12\int_{\bd B} xRf^2    \ud \mathcal H^{N-1}
	\end{align*}
	where $ |R(x)|\leq C(N)$ for every $x\in\bd B.$
	We can then estimate
	\[     \left\lvert\int_{\bd B} f(x)x\ud \mathcal H^{N-1}  \right\rvert\leq C\|f\|_{L^2}^2 \] provided $\|f\|_{C^1}\le \delta$ and the conclusion follows recalling $\nu_B(x)=x$.
	
\end{oss}


\section{Uniform $L^1-$estimate on the discrete flow}
\label{section L^1-estimate}
In this section we give the precise definition of the discrete volume preserving flow in the flat torus and we study some of its properties. In particular, we prove Proposition \ref{uniform L1 estimate} that will play a crucial role in the proof of our main result.

\subsection{Discrete volume preserving mean-curvature flow}

Let $E \ne \emptyset$ be a measurable subset of $\T^N$. In the following we will always assume that $E$ coincides with its Lebesgue representative.
Fixed $h>0$, $m\in (0,1)$, we consider the minimum problem
\begin{equation}
	\min \left \{ P(F)+\frac{1}{h} \int_{F} \sd_E(x) \ud x : F \subset \T^N, \,\lvert F \rvert =m \right \},
	\label{pb_min_1}
\end{equation}
where $\sd_E(x):= \dist_E (x)-\dist_{\T^N \setminus E}(x)$ is the signed distance from the set $E$. Observe that the minimum problem \eqref{pb_min_1} is equivalent to the problem 
\begin{equation*}
	\min \left \{ P(F)+\frac{1}{h} \int_{F \triangle E} \dist_{\bd E}(x) \ud x : F \subset \T^N, \,\lvert F \rvert =m \right \}.
\end{equation*}
For every $F \subset \T^N$, we set
\begin{equation}
	J_h^E(F):=P(F)+\frac 1h\int_{F\triangle E}\dist_{\bd E}(x) \ud x =: P(F)+\frac{1}{h} \mathcal D(F,E),
	\label{def funzionale}
\end{equation}
with a little abuse of notation we will sometimes denote by $J_h^E$ also the functional
\[F \mapsto P(F)+\frac{1}{h} \int_{F} \sd_E(x) \ud x\]
and, when no ambiguity arises, we will write  $J_h$ instead of $J_h^E$.

By induction we can now define the \emph{discrete-in-time, volume preserving mean curvature flow} $(E_h^n)_{n\in \N}$ and we will refer to it as the \textit{discrete flow.}
Let $E_0 \subset \T^N$ be a measurable set such that $\lvert E_0 \rvert = m$, we define $E_h^1$ as a solution of \eqref{pb_min_1} with $E_0$ instead of $E$, i.e.
\[E_h^1 \in \textnormal{argmin}\left\{ P(F)+\frac{1}{h} \int_{F} \sd_{E_0}(x) \ud x : F \subset \T^N, \,\lvert F \rvert =m \right \}.\] Assume that $E_h^k$ is defined for $1 \le k \le n-1$, we define $E_h^{n}$ as a solution of \eqref{pb_min_1} with $E$ replaced by $E_h^{n-1}$, i.e.
\[E_h^n \in \textnormal{argmin}\left\{ P(F)+\frac{1}{h} \int_{F} \sd_{E_h^{n-1}}(x) \ud x : F \subset \T^N, \,\lvert F \rvert =m \right \}.\]

\begin{oss}
	\label{oss limitatezza perimetro}
	We start by remarking that the sequence of the perimeters along the discrete flow is non-increasing. Indeed, from the minimality of $E_h^n$ and considering $E^{n-1}_h$ as a competitor we obtain
	\[ P(E_h^n)\le P(E_h^n)+\frac{1}{h}\int_{E^{n-1}_h \triangle E^n_h} \dist_{\bd E^{n-1}_h}(x) \ud x\le P(E^{n-1}_h). \]
	From this simple remark we observe that, even if the starting set of the flow $E_0$ is not of finite perimeter, the perimeters of the sets $E_h^n$ are uniformly bounded by a constant that only depends on the dimension $N$, the fixed volume $m$ and $h$. Given any set $E_0$ of volume $m$, consider the cube $Q_m$ of the same volume. From the minimality of $E_h^1$ and using $Q_m$ as a competitor we obtain 
	\begin{align*}
		P(E^1_h)&\le P(Q_m)+\frac 1h \int_{E_0\triangle Q_m}\dist_{\bd E_0}(x) \ud x-\frac 1h \int_{E_0\triangle E_1^h}\dist_{\bd E_0}(x) \ud x\\
		&\le P(Q_m) + \dfrac 1h\int_{\T^N}\sqrt{N}=C(N,m,h),
	\end{align*}
	where we estimated $\dist_{\bd E_0}\le \text{diam}(\T^N)=\sqrt N$.
\end{oss}

We recall some preliminary results that can be found in \cite{MPS}. If not otherwise stated, their original proofs can be easily adapted to the periodic case, the major difference being that in our case we work in the flat torus, which is compact, thus simplifying some arguments.
First of all, we observe that  that the problem \eqref{pb_min_1} admits solutions via the standard method of the calculus of variations.

The regularity properties of the discrete flow are investigated in the following proposition. Some of the results are classical, others follow from \cite[Proposition 2.3]{MPS}.

\begin{prop}
	\label{proprieta flusso}
	Let $h,$ $m$, $M>0$ and let $E \subset \T^N$ be a set with $\lvert E \rvert =m$ and $P(E)\le M$.  Then, any solution $F \subset \T^N$ to \eqref{pb_min_1} satisfies the following regularity properties:
	\begin{itemize}
		\item[i)] There exist $c_0=c_0(N)>0$ and a radius $r_0=r_0(m,h,N,M)>0$ such that for every $x \in \bd^* F$ and $r \in (0,r_0]$ we have
		\begin{equation*}
			\lvert B_r(x) \cap F \rvert \ge c_0 r^N \quad \text{and} \quad \lvert B_r(x) \setminus F \rvert \ge c_0 r^N.
		\end{equation*}
		In particular, $F$ admits an open representative whose topological boundary coincides with the closure of its reduced boundary, i.e. $\bd F= \overline{\bd^* F}$.
		
		\item[ii)] There exists $\Lambda=\Lambda(m,h,N,M)>0$ such that $F$ is a $\Lambda$-minimizer of the perimeter, that is
		\begin{equation*}
			P(F) \le P(F') + \Lambda \lvert F \triangle F' \rvert
		\end{equation*}
		for all measurable set $F' \subset \T^N$.
		
		\item[iii)] The following Euler-Lagrange equation holds: there exists $\lambda \in \R$ such that for all $X \in C^1_c(\T^N,\T^N)$ we have
		\begin{equation}
			\int_{\bd^*F}\dfrac {\sd_{ E}}h X\cdot \nu_F \ud \mathcal{H}^{N-1}+\int_{\bd^* F} \div_{\tau} X \ud \mathcal{H}^{N-1}=\lambda\int_{\bd^* F} X\cdot \nu_F  \ud \mathcal{H}^{N-1}.
			\label{euler-lagrange equation}
		\end{equation}
		
		\item[iv)] There exists a closed set $\Sigma$,
		whose Hausdorff dimension is less than or equal to $N-8$, such that $\bd^* F= \bd F \setminus \Sigma$ is an $(N-1)$-submanifold of class $C^{2,\alpha}$ for all $\alpha \in (0,1)$ with
		\begin{equation*}
			\lvert H_{F}(x) \rvert \le \Lambda, \quad \text{for all } x \in \bd F \setminus \Sigma.
		\end{equation*}
		
		\item[v)] There exists $k_0=k_0(m,h,N,M)\in \N$ and $s_0=s_0(m,h,N,M)>0$ such that $F$ is made up of at most $k_0$ connected components having mutual Hausdorff distance at least $s_0$.
		
	\end{itemize}
	\label{prop regularity discrete flow}
\end{prop}

The following result characterizes the stationary sets of the discrete scheme.
The last assertion of the proposition is a technical result that will be employed in the proof of Lemma~\ref{L^inf estimate}.
\begin{prop}
	\label{prop 3.1 rivisitata}
	Every stationary set $E$ for the discrete flow is a critical set of the perimeter.
	\\
	Viceversa, if $E$ is a \rosso{regular} critical set of the perimeter, then there exists $h^*=h^*(E)>0$ such that, for every $h<h^*$, the volume preserving discrete flow starting from $E$ is unique and given by $E_h^n=E.$
	Moreover, if $E$ is a strictly stable set then it is also the unique volume-constrained minimizer of the functional
	\begin{equation*}
		\tilde J_h(F):= P(F)+\frac{1}{h}\int_F \dist_E(x) \ud x .
	\end{equation*}
\end{prop}
\begin{proof}
	The first statement is an immediate consequence of \eqref{euler-lagrange equation}. Since $E$ is a stationary point for the discrete flow, it satisfies
	\begin{equation*}
		\int_{\bd^* E}\div_{\tau} X \ud \mathcal{H}^{N-1} =
		\lambda \int_{\bd^* E} X \cdot \nu_E\ud \mathcal{H}^{N-1}
	\end{equation*}
	for all $X \in C^1_c(\T^N,\T^N)$, i.e. $E$ is a critical point for the perimeter.

	The second part follows using the same argument of the proof of \cite[Proposition~3.2]{MPS}. Indeed, recall that the second variation has the following expression
	\[ \bd^2 J_h(E)[\varphi]=\int_{\bd E} |\nabla \varphi|^2 + \left(  \dfrac 1h -|B_E|^2 \right) \varphi^2\udH, \]
	which is positive if $h$ is small enough. Then we procede as in the proof of  \cite[Proposition~3.2]{MPS}.

	Analogously, we prove that $E$ is the unique volume-constrained minimizer of $\tilde J_h$. Firstly, observe that, by Theorem \ref{coroll 1.2}, $E$ is a strict local $L^1$-minimizer of the perimeter and it is a global minimizer of the second term in $\tilde J_h$. Therefore, there exists $\varepsilon>0$ such that \[\tilde J_h(E) <\tilde J_h(F) \] for all measurable set $F$ such that $|F|=|E|$ and $|E \triangle F|\le \varepsilon$, i.e. $E$ is an isolated local minimizer for $\tilde J_h$ in $L^1$ with the volume constraint, with minimality neighbourhood uniform with respect to $h$. Now, given any sequence $(h_n)_{n \in \N}$ going to zero, let $F_n$ be a volume constrained minimizer of $J_{h_n}$; we then easily deduce that $|E \triangle F_n| \to 0$ as $n \to \infty$, and therefore, for $n$ large enough, $|E\triangle F_n|\le \varepsilon$. The strict minimality of $E$ therefore implies  $F_n=E$.

\end{proof}


\subsection{Uniform $L^1$ estimate}

In this subsection we prove a uniform $L^1-$estimate on the discrete flow starting from an initial set $E_0$ sufficiently \virg{close} to a strictly stable set of the perimeter. We will devote the next subsection to a discussion upon the hypotheses of the estimate. Before we recall the definition of Hausdorff distance and some of its properties, for a complete reference see e.g. \cite[Section~4.4]{AmT}, \cite[Section~10.1]{MoSo}.

Given a set $C \subset \T^N$, we denote by $(C)_{\delta}$  the $\delta$ fattened of $C$, that is the set 
\[\{x \in \T^N: \dist_C(x)\le \delta  \}.\]
Let $C_1,$ $C_2 \subset \T^N$ be closed sets, we define the \emph{Hausdorff distance} between $C_1$ and $C_2$ as 
\[\ud_H(C_1,C_2):= \inf \left \{\rho>0: C_1 \subset (C_2)_{\rho}, \, C_2 \subset (C_1)_{\rho} \right \}.\]
Given $C_n,$ $C$ closed sets in $\T^N$, we say that $(C_n)_{n \in \N}$ converges to $C$ in the Hausdorff distance and we write $C_n \overset{H}{\to} C$, if $\ud_H(C_n,C)\to 0$ as $n \to \infty$. We recall that the space of closed subsets of a compact set equipped with the Hausdorff metric is compact (see e.g \cite[Theorem~4.4.15]{AmT}  or \cite[Proposition~10.1]{MoSo}) and also that the convergence in the Hausdorff distance is equivalent to the uniform convergence of the respective distance functions, i.e. 
\[ C_n\overset{H}{\to}C\quad \iff \quad \dist_{C_n}\to \dist_{C}\quad \text{uniformly}. \]
In the following, given two open smooth sets $E_1$, $E_2$, we will denote by $\ud_H(E_1,E_2)$ the Hausdorff distance between their closures.

\begin{lemma}
	\label{L^inf estimate}
	Let $E \subset \T^N$ be a strictly stable set and let $\varepsilon>0$. Then, there exist $\delta=\delta(\varepsilon, E)>0$ and $h^*=h^*(E)>0$ such that, for every $h < h^*$ and for every set $E_0$ satisfying
	\[|E_0| =|E|, \qquad \ud_H(\overline{E}_0,\overline{E}) \le\delta,\]
	we have
	\[|E \triangle F| \le \varepsilon, \]
	where $F$ is a solution of \eqref{pb_min_1} with $E_0$ replacing $E$.
\end{lemma}

\begin{proof} 
	Let $h^*=h^*(E)$ be the constant given by Proposition \ref{prop 3.1 rivisitata} so that, for every $h< h^*$, $E$ is the unique volume-constrained global minimizer of the functional 
	\begin{equation}
		\label{E_min_0}
		\tilde J_h(G):= P(G)+\frac{1}{h}\int_G \dist_{E}(x) \ud x.
	\end{equation}
	Fix $h<h^*$ and let $(E_n)_{n \in \N}$ be a sequence of sets satisfying
	\begin{equation}
		|E_n|=|E|, \qquad \overline{E}_n \overset{H}{\to} \overline{E}.
		\label{ipotesi}  
	\end{equation}
	Consider $F_n$ a solution of \eqref{pb_min_1} with $E_n$ replacing $E$. We claim that
	\begin{equation*}
		F_n \overset{L^1}\to E.
	\end{equation*} 
	If we prove the claim, the conclusion easily follows.

	First, Remark \ref{oss limitatezza perimetro} ensures that  $(F_n)_{n \in \N}$ is a sequence of  sets with uniformly bounded perimeters, with the bound depending only on $N,m,h$. Therefore, there exist $F$ a set of finite perimeter such that $|F|=m$ and a (unrelabelled) subsequence of $(F_n)_{n \in \N}$ such that 
	\begin{equation*}
		F_n\overset{L^1}\to F . 
	\end{equation*}
	Now, let $K$ be a compact subset of $\T^N$ such that, up to a subsequence, we have
	\begin{equation*}
		\overline{E_n^c} \overset{H}{\to} K.
	\end{equation*}
	From the second property in \eqref{ipotesi} we easily deduce that $(\overline E)^c \subset K$, and therefore $K^c\subset \overline E$. In particular, this inclusion implies that
	\begin{equation*}
		\int_{K^c} \dist_K(x) \ud x= \int_E \dist_K(x) \ud x \ge \int_G \dist_K(x) \ud x
	\end{equation*}
	for every $G \subset \T^N$.
	Setting 
	\[\bar{J}_h(G):=P(G)+\frac{1}{h} \int_G (\dist_E(x)-\dist_{K}(x))\ud x, \]
	from the previous remark and from the fact that $E$ is the unique minimizer of \eqref{E_min_0}, we have
	\begin{align*}
		\bar{J}_h(G)&= \tilde J_h(G)-\frac{1}{h}\int_G  \dist_K(x)\ud x\\
		&>\tilde J_h(E)-\frac{1}{h}\int_G  \dist_K(x)\ud x\\
		&\ge \tilde J_h(E)-\frac{1}{h}\int_E \dist_K(x)\ud x=\bar{J}_h(E),
	\end{align*}
	for any measurable set $G\subset \T^N$ with $|G|=|E|$. Finally, we obtain
	\begin{align*}
		\bar{J}_h(F)&=P(F)+\frac 1h \int_{F} \left(\dist_E( x)-\dist_{K}(x)\right) \ud x\\
		&\le \liminf_{n\to \infty} P(F_n) +\frac 1h \int_{F} \left(\dist_{E}( x)-\dist_{K}(x)\right) \ud x\\
		&=\liminf_{n \to \infty} \left( P(F_n)+\frac 1h \int_{F_n} \left(\dist_{E_n}( x)-\dist_{E_n^c}(x)\right) \ud x \right)\\
		&\le \liminf_{n\to \infty} \left( P(E)+\frac 1h \int_{E} \left(\dist_{E_n}( x)-\dist_{E_n^c}(x)\right) \ud x \right)\\
		&=P(E)-\frac{1}{h}\int_E \dist_{K}(x)\ud x=\bar{J}_h(E)
	\end{align*}
	where we exploited the lower-semicontinuity of the perimeter and the minimality of $F_n$.
	Since $E$ is the unique volume-constrained minimizer of $\bar{J}_h$, the set $F$ must coincide with $E$ and this concludes the proof.
\end{proof}

\begin{oss}
	\label{solo_una_inc}
	We remark that under the hypotheses of Lemma \ref{L^inf estimate} we could have just assumed the one-sided inclusion
	\[\overline{E}_0 \subset (E)_{\delta^*} \]
	instead of \[\ud_H(\overline{E}_0,\overline{E})\le \delta\]
	for a suitable $\delta^*\le \delta$. Indeed, let $\delta_n \to 0$ and $E_n\subset (E)_{\delta_n}$ such that $|E_n|=|E|$. We prove that $\overline{E}_n$ converges to $\overline{E}$ in the sense of Kuratowski (and thus with respect to Hausdorff).
	Let $(x_n)_{n \in \N}$ be a sequence such that $x_n\in \overline{E}_n$ and $x_n \to y$. For every $n\in\N$, there exists $y_n\in E$ such that  $|x_n-y_n|\le \delta_n.$ Therefore, for any $\varepsilon>0$ there exists $n_0$ such that, for $n\ge n_0$, we have
	\[ |y_n-y|\le |y_n-x_n|+|x_n-y|\le     \delta_n+\varepsilon,\]
	that is $y_n\to y$. Since $(y_n)_{n \in \N}\subset E$, we have $y\in \overline{E}$.
	
	Fix now $y\in \overline{E}$. Assume by contradiction that there exists $\delta>0$ such that $\dist_{E_n} (y)> \delta$, i.e. it doesn't exist a sequence of elements in $\overline{E}_n$ converging to $y$. From this (and up to subsequences) it follows  \[E_n\subset (E)_{\delta_n}\setminus B_\delta (y)\qquad \forall n \in \N.\]
	Thus we have
	\begin{align*}
		m&= \lim_{n \to \infty}|E_n|\le \lim_{n \to \infty}|(E)_{\delta_n}\setminus B_\delta (y)|\\
		&\le \lim_{n \to \infty} |(E)_{\delta_n}\setminus \left(  B_\delta(y)\cap E \right)|\\
		&=\lim_{n \to \infty}|(E)_{\delta_n}|-| B_\delta(y)\cap E| = m - | B_\delta(y)\cap E|
	\end{align*}   
	which is a contradiction.
\end{oss}

We are now able to prove the main estimate that will be used in the proof of Proposition \ref{teorema convergenza a meno di traslazioni}.

\begin{prop}[Uniform $L^1-$estimate] 
	\label{uniform L1 estimate}
	Let $E \subset \T^N$ be a strictly stable set. Then, for every $\varepsilon >0$ there exist $\delta^*=\delta^*(\varepsilon, E) >0$ and $h^*=h^*(E)>0$ with the following property: for every $h <  h^*$, if $E_0$ is a measurable set such that 
	\begin{equation*}
		|E_0| =|E|, \qquad \overline{E}_0\subset(E)_{\delta^*},
	\end{equation*} 
	then the discrete flow $(E_h^n)_{n \in \N}$ starting from $E_0$ satisfies \[\alpha(E,E^n_h) \le \varepsilon\]
	for every $n \in \N$ .
	
\end{prop}

\begin{proof}
	Fix $h< h^*$, where $h^*=h^*(E)$ is the constant given by Lemma \ref{L^inf estimate} and let $\sigma=\sigma(E)$, $C=C(E)$ be the constants of Theorem \ref{coroll 1.2}.
	Moreover, let
	$\delta:= \delta(\sigma,E)$ be the constant given by Lemma \ref{L^inf estimate} with  $\sigma$ replacing  $\varepsilon$. 
	Set $\delta^*\le \delta$ to be chosen later and consider $E_0$ such that 
	\[|E_0| =|E|, \qquad \overline{E}_0 \subset (E)_{\delta^*}.\]
	Recall that, from Remark \ref{solo_una_inc} and from the hypothesis $\overline{E}_0 \subset (E)_{\delta^*}$, without loss of generality, we can assume $\ud_H(\overline{E}_0,\overline{E}) \le \delta^*$.
	Moreover, by the regularity of $E$, we can also suppose $\alpha(E_0,E)\le \tilde C\delta^*$, for a suitable constant $\tilde C>0$ that only depends on $E$.
	From Lemma \ref{L^inf estimate} we have that
	\begin{equation}
		|E_h^1\triangle E| \le \sigma.
		\label{siamo vicini}
	\end{equation}
	Let $x_0$ be such that $\alpha (E_0,E)= \lvert E_0 \triangle (E+x_0) \rvert$. By choosing $E+x_0$ as a competitor for the minimality of $E_h^1$ and estimating dist$_{\bd E_0}\leq \text{diam}(\T^N)= \sqrt{N}$, we find
	\[ P(E_h^1)-P(E)\leq \dfrac{1} h\int_{E_0 \triangle (E+x_0)} \dist_{\bd E_0}(x) \ud x \le \dfrac{\sqrt{N}} h \alpha (E_0, E)\le \frac{\sqrt{N}}{h} \tilde C\delta^*.  \]
	By \eqref{siamo vicini}, we can apply Theorem \ref{coroll 1.2} and the previous estimate to obtain
	\begin{align*}
		\alpha(E_h^1,E)\leq \dfrac 1{\sqrt C}\sqrt{P(E_h^1)-P(E)}\le \dfrac 1{\sqrt C}\sqrt{\frac{\sqrt{N}}{h} \alpha(E,E_0)}\leq \dfrac 1{\sqrt C}\sqrt{\frac{\sqrt{N}}{h}  \tilde C \delta^* } \le \min \{ \sigma, \delta,\varepsilon\},
	\end{align*}
	where we have chosen $\delta^*$ such that $\delta^* \le C h\left (\min \{ \sigma, \delta,\varepsilon\} \right )^2/ (\tilde C\sqrt{N}).$
	Since $E_h^1$ is a $\Lambda-$minimizer and $E$ is regular, up to taking $\delta^*$ smaller, the classical regularity theory for $\Lambda-$minimizers (see Theorem \ref{teorema convergenza C^1}) implies
	\[\ud_H(\bd E_h^1, \bd E+x_1) \le \delta, \]
	where $x_1$ is such that $\alpha(E_h^1,E)=|E_h^1 \triangle (E+x_1)|$.
	
	Now we iterate the procedure: by induction, suppose that 
	\begin{equation}
		\alpha(E_h^{n-1},E) \le \min \{ \sigma, \delta,\varepsilon\}, \qquad \ud_H(\bd E_h^{n-1}, \bd E+x_{n-1}) \le \delta
		\label{inequalities above}
	\end{equation}
	where $x_{n-1}$ is such that $|E_h^{n-1} \triangle (E +x_{n-1})|=\alpha(E_h^{n-1},E)$. Observe that the second inequality in \eqref{inequalities above}  implies that $\ud_H(\overline{E}_h^{n-1},\overline{E}+x_{n-1})\le \delta$, therefore $E_h^{n-1}$ and $E+x_{n-1}$ satisfy the hypotheses of Lemma \ref{L^inf estimate} and thus
	\[|E_h^n \triangle (E+x_{n-1})| \le \sigma.\]
	Observe that by definition $ \alpha(E_h^n,E+x_{n-1})=\alpha(E_h^n,E)$. Now, by Theorem \ref{coroll 1.2} and the monotonicity of the perimeters along the discrete flow we obtain 
	\begin{align*}
		\alpha(E_h^n,E)&\leq \dfrac 1{\sqrt C}\sqrt{P(E_h^n)-P(E)}\\
		&\leq \dfrac 1{\sqrt C}\sqrt{P(E_h^1)-P(E)} \\
		&\leq \dfrac 1{\sqrt C}\sqrt{\dfrac{\sqrt N}h  \tilde C\delta^* } \le \min \{ \sigma,\delta,\varepsilon\}.
	\end{align*}
	Again, thanks to the choice of $\delta^*$, the hypotheses of Theorem \ref{teorema convergenza C^1} are satisfied and thus 
	\[ \ud_H(\bd E_h^n,\bd E+x_n)\le \delta, \]
	where $x_n$ is such that $\alpha(E_h^n,E)=|E_h^n\triangle (E+x_n)|$. This concludes the proof.
\end{proof}

\subsection{Some remarks on the hypothesis of the $L^1-$estimate}
\label{sec_contr}

In this subsection we show that Proposition \ref{uniform L1 estimate} does not hold if we weaken the hypothesis of closeness in the Hausdorff distance between the starting set $E_0$ and the strictly stable set $E$. In particular, we prove that the sole hypothesis of closeness in $L^1$ and in perimeter is not enough. We remark that a  modification of this example yields the same result in $\R^N$.

Fix $h>0$ and $G \subset \T^N$. Recall that, for any set $F\subset \T^N$ such that $|F|=|G|$, we have set
\begin{equation}\label{1998}
	J_h^G(F):=P(F)+\frac 1h\mathcal \int_{F \triangle G} \dist_{\bd G}(x) \ud x.
\end{equation}

\begin{prop}
	\label{first counterexample}
	There exist $m>0$ and a sequence $(E_n)_{n\in \N}\subset \T^N$ with the following properties: $|E_n|=m$ for every $n \in \N$, $P(E_n)$ is uniformly bounded and, letting $F_n$ be any volume-constrained minimizer of \eqref{1998} with $E_n$ instead of $G$, we have
	\[ E_n\overset{L^1}{\to} E,\quad P(E_n) \to P(E) \quad \text{but}\quad F_n\overset{L^1}{\to} F, \]
	where $E$ is a lamella  and $F$ is such that $|E\triangle F|>0.$
\end{prop}

\begin{proof} 
	Let $m>0$ such that the ball  of volume $m$ has perimeter strictly less than the one of the lamella of the same volume; we remark that for every smaller volume $m'\le m$ the same property holds.
	Let $E$ be a lamella of measure $m$, recall that $E$ is a strictly stable set of the perimeter in $\T^N$. 
	From the assumption on $m$ it follows that $E$ is only a local minimizer of the perimeter and not a global one.\\
	\textbf{Step 1.} Firstly, we construct a sequence $(E_n)_{n \in \N}$ such that $E_n \to E$ in $L^1$ and $\bd E_n \to \T^N$ in the Hausdorff distance.
	We define $E_n$ by adding to $E$ some balls contained in $\mathbb T^N\setminus E$ and of overall small volume, and by subtracting to $E$ balls contained in $E$ with the same overall volume. 
	
	Recall that $\T^N=[0,1]^N/\Z^N$.
	In the following, with a little abuse of notation, we will identify $\T^N$ and $[0,1)^N$.
	We define
	\begin{align*}
		I_n:&=\left \{\underline k =(k_1,\ldots, k_N) \in \Z^N: \, 0\le k_i \le 2^{n}-1 \quad \forall i =1,\ldots N\right\},\\
		\mathcal{P}_n:&=\left \{Q_{n,\underline k}:=\Big[0,\frac{1}{2^n}\Big)^N+\frac{\underline k}{2^n}:\, \underline k \in I_n   \right \},
	\end{align*} 
	for every $n \in \N$. Up to choosing $m$ smaller, we can assume that $m=1/2^s$ for some $s \in \N$. Moreover, we can suppose, up to translations, that $E=[0,1)^{N-1} \times (0,1/2^s)$, thus for $n \ge s$ we have
	\[E =\textnormal{Int}\left ( \bigcup_{ \underline k  \in I_n, \,\,0\le k_N \le 2^{n-s}-1} Q_{n,\underline k}\right), \]
	where $\textnormal{Int}(\cdot)$ denotes the interior of a set in $\T^N$.
	For every $n\ge s$ and 
	$\underline k  \in I_n$, we consider the balls $B_{n,\underline k} \subset Q_{n,\underline k}$  centered in the center of the cube $Q_{n,\underline k}$ and of radius $r_{n,\underline k}$ chosen in such a way that
	\begin{equation}
		\label{misura_uguale}
		\left\lvert\bigcup_{\underline k   \in I_n,\, \,0\le k_N \le 2^{n-s}-1} B_{n,\underline k}\right\rvert= \left\lvert \bigcup_{\underline k   \in I_n, \,\,2^{n-s}\le k_N \le 2^{n}-1} B_{n,\underline k} \right\rvert.
	\end{equation}
	Moreover, we can also take the radii $r_{n,\underline k}$ sufficiently small so that
	\begin{equation}\lim_{n \to \infty}\left\lvert\bigcup_{\underline k  \in I_n} B_{n,\underline k}\right\rvert = 0, \qquad \lim_{n \to \infty}P   \left(\bigcup_{\underline k   \in I_n } B_{n,\underline k}\right)=0.
		\label{misura_va_zero}
	\end{equation}
	Set now
	\begin{align*} 
		A_n&:=\bigcup_{\underline k   \in I_n ,\,\,0\le k_N \le 2^{n-s}-1} B_{n,\underline k}\subset \textnormal{Int}\left (\bigcup_{\underline k   \in I_n ,\,\,0\le k_N \le 2^{n-s}-1} Q_{n,\underline k}\right)= E,\\ C_n&:=\bigcup_{\underline k   \in I_n ,\,\,2^{n-s}\le k_N \le 2^{n}-1} B_{n,\underline k}\subset \bigcup_{\underline k   \in I_n ,\,\,2^{n-s}\le k_N \le 2^{n}-1} Q_{n,\underline k}\subset \T^N\setminus E.
	\end{align*}
	Define $E_n=(E\cup C_n)\setminus A_n$ and observe that, by \eqref{misura_uguale}, we have $|E_n|=|E|$. Now, by \eqref{misura_va_zero}, we also obtain 
	\[ E_n\overset{L^1}{\to} E \quad \textnormal{ and } \quad P(E_n) \to P(E).\]
	Observe that, from the definition of $A_n$ and $C_n$, we have that
	\[(\bd A_n)_{\sqrt{N}/2^n} \cup (\bd C_n)_{\sqrt{N}/2^n}=\T^N\]
	and therefore the set $\bd E_n= \bd E \cup \bd C_n \cup \bd A_n$ converges in the Hausdorff metric to the whole $\T^N$  as $n\to +\infty$. Therefore we have constructed a sequence $(E_n)_{n \in \N}$ that satisfies
	\begin{equation}
		E_n\overset{L^1}{\to} E,\qquad  P(E_n) \to P(E), \qquad \bd E_n\overset{H}{\to} \T^N.
		\label{insiemi controesempio}
	\end{equation}
	
	\textbf{Step 2.} Let $E_n$ be the sets previously defined.
	We consider the space $X=\{  F\subset \T^N\,:\, F \text{ is measurable} \}$ endowed with the $L^1-$distance, i.e. $\dist_{L^1}(F,G)=|F\triangle G|$ for every $F,G\in X$. We extend our functional in the following way 
	\begin{equation*}
		\tilde J^E_h(F):=\begin{cases} J_h^E(F) &\text{if } P(F)<\infty,\ |F|=m,\\[1ex]
			+\infty  &\text{otherwise}\end{cases}
	\end{equation*}
	and we set $J_n:=\tilde J_h^{E_n}$. We then prove the $\Gamma-$convergence of the functionals $J_n$ to the perimeter functional in $X$, that is
	\begin{equation}
		\Gamma(X)-\lim_{n \to \infty}J_n=P.
		\label{gamma-convergence}
	\end{equation}
	
	We can clearly restrict ourselves to consider sets of finite perimeter and volume $m$, otherwise the result is trivial. For any given set $F$ of measure $m$ and finite perimeter we choose the sequence constantly equal to $F$ as a recovery sequence for $F$. Indeed, by~\eqref{insiemi controesempio} we have
	\[ J_n(F)=P(F)+\frac 1h \int_{F\triangle E_n} \dist_{\bd E_n}\to P(F).\]
	We now prove the $\liminf$ inequality. Given a sequence $F_n$ that converges to $F$ in $L^1$, by the $L^1-$semicontinuity of the perimeter, we have
	\[ P(F)\le \liminf_{n\to \infty} P(F_n)\le \liminf_{n\to \infty} \left(    P(F_n)+\frac 1h \int_{F_n\triangle E_n}\dist_{\bd E_n}     \right)\]
	and thus \eqref{gamma-convergence} is proved.
	Therefore, thanks to the equi-coercivity of the functionals $J_n$, any sequence of volume-constrained global minimizers of $J_n$ converges in $L^1$, up to a subsequence, to a volume-constrained global minimizer of the perimeter in the torus. Let $(F_n)_{n \in \N}$ be a sequence of global minimizers of the functional $J_n$ and let $F$ be such that $F_n\to F$ in $L^1$. We know that $F$ is a global minimizer of the perimeter and that by the choice of $m$ the lamella is not a global minimizer. Therefore it must hold $|E\triangle F|>0$.
\end{proof}


\section{Convergence of the flow}
\label{sezione convergenza}
In this section, we will prove the main result of the paper concerning the convergence of the discrete flow that mainly relies on Proposition \ref{uniform L1 estimate}.


\subsection{Convergence of the flow up to translations}

We start by recalling \cite[Lemma~3.6]{MPS}: it will be used in the proof of the following proposition. 

\begin{lemma}
	Let $(E_h^n)_{n \in \N}$ be a volume preserving discrete flow starting from $E_0$ and let $E_h^{k_n}$ be a subsequence such that $E^{k_n}_h+\tau_n\to F$ in $L^1(\T^N)$ for some set $F$ and a suitable sequence $(\tau_n)_{n \in\N}\subset\T^N.$ Then $\dist_{\bd E^{k_n-1}_h}(\cdot+\tau_n)\to \dist_{\bd F}$ uniformly. 
	\label{lemma 3.5n}
\end{lemma}

In the following proposition we characterize the long-time behaviour up to translations of the discrete mean curvature flow in the flat torus starting near a regular strictly stable set.

\begin{prop}
	Let $E \subset \T^N$ be a strictly stable set. Then there exist $\delta^*=\delta^*(E)>0$ and $h^*=h^*(E)>0$ with the following property: if $h< h^*$ and $E_0\subset \T^N$ is a set of finite perimeter satisfying
	\[|E_0| =|E|, \qquad \overline E_0\subset (E)_{\delta^*},\]
	then, for every discrete flow $(E_h^n)_{n\in\N}$ starting from $E_0$, there exists a sequence of translations $\tau_n\in \T^N$ such that
	\[ E_h^n+\tau_n \to E  \quad \text{in} \quad  C^k, \quad \forall k \in \N.\]
	\label{teorema convergenza a meno di traslazioni}
\end{prop}

\begin{proof} 	
	Let $\varepsilon>0$ be sufficiently small and let $\delta^*=\delta^*(\varepsilon,E)$, $h^*=h^*(E)$ be the constants given by Proposition \ref{uniform L1 estimate}.
	Fix $E_0$ an initial set satisfying $|E|=|E_0|$ and $\overline E_0\subset (E)_{\delta^*}$. It is enough to show that any (unrelabelled) subsequence of the discrete flow starting from $E_0$ admits a further subsequence converging in $C^k$ and up to translations to $E$. 
	We divide the proof into three steps.\\
	\textbf{Step 1.} (Existence and regularity of a limit point) From Proposition \ref{prop regularity discrete flow} we remark that, for $n \ge 1$, the sets $E_h^n$ are uniform $\Lambda-$minimizers with uniformly bounded, non-increasing perimeters. Therefore, by the compactness of (uniform) $\Lambda-$minimizers, we can conclude that there exists a subsequence $(E_h^{k_n})_{n \in \N}$ and a $\Lambda-$minimizer $E_h^{\infty}$ such that
	\[ E_h^{k_n}\overset{L^1}{\to} E_h^{\infty},\quad  P(E_h^{k_n}) \to P(E_h^{\infty}), \quad  \sd_{ E_h^{k_n-1}}\to \sd_{ E_h^{\infty}}\ \text{uniformly}. \]
	Let $G$ be a set of finite perimeter such that $|G|=m$. By the minimality of $E_h^{k_n}$ we have
	\[ P(E_h^{k_n})+\dfrac 1h \int_{E_h^{k_n}} \sd_{ E_h^{k_n-1}}(x) \ud x\leq P(G)+\dfrac 1h \int_G \sd_{ E_h^{k_n-1}}(x) \ud x  \]
	and, taking the limit as $n\to \infty$, we obtain
	\[ P(E_h^{\infty})+\dfrac 1h\int_{E_h^{\infty}} \sd_{ E_h^{\infty}}(x) \ud x\leq P(G)+\dfrac 1h \int_G \sd_{ E_h^{\infty}}(x) \ud x.  \]
	We have thus proved that $E_h^{\infty}$ is a fixed point for the discrete flow and thus, by Proposition \ref{prop 3.1 rivisitata}, it is a critical point for the perimeter.  
	
	Let $\tau_\infty\in\text{argmin}_x|(E_h^\infty+x)\triangle E|$. 
	By Proposition \ref{uniform L1 estimate} we have $\alpha(E,E_h^{k_n}) \le \varepsilon$ for every $n \in \N$. Now, up to taking $\varepsilon$ smaller, Theorem \ref{teorema convergenza C^1} and  the smoothness of $E$, yields  both the $C^{1,\beta}$-closeness between $E_{h}^{\infty}+\tau_\infty$ and $E$, and the $C^{1,\beta}$ regularity of $E_h^\infty+\tau_\infty$ (and thus of $E_h^\infty$), for every $\beta\in (0,1)$. From Proposition \ref{prop regularity discrete flow} \textit{(iv)}  it follows that  $E_h^\infty$  is of class $C^{2,\beta}$, therefore we conclude that $E_h^\infty$ has constant classical mean curvature and thus it is of class $C^\infty$. To conclude, the smoothness of $E^\infty_h$ allow us to use Theorem \ref{teorema convergenza C^1} to improve the convergence of the subsequence to
	\begin{equation}
		\label{eq step 2}E_h^{k_n}\to E^\infty_h\quad \text{in}\quad C^{1,\beta}
	\end{equation}
	and to ensure that the sets $E^{k_n}_h$ are of class $C^{1,\beta}$ for $n$ large enough.\\
	\textbf{Step 2.} (Convergence in $C^{2,\beta}$ of the flow and $C^{2,\beta}-$closeness to $E$) In this step we we will prove that $E^\infty_h$ is $C^{2,\beta}-$close to $E$ and that the convergence of $E^{k_n}_h$ to $E^{\infty}_h$ is in $C^{2,\beta}$. Without loss of generality, we assume that  $\alpha(E,E^\infty_h)=|E\triangle E^\infty_h|$ so that the translation introduced by the previous step does not appear.

	First of all we remark that, owing to the compactness of $\bd E^\infty_h$, it suffices to show that the result holds locally. By a compactness argument and the definition of convergence of sets in $C^{1,\beta}$ (Definition \ref{definizione convergenza C^1}), up to rotations and relabelling the coordinates, we can find a cylinder $C=B'\times (-L,L)$, where $B'\subset \R^{N-1}$ is a ball centred at the origin, and functions $f_\infty,f_n\in C^{1,\beta}(B';(-L,L))$ 	describing locally $\bd E^\infty_h\cap C$ and $\bd E^{k_n}_h\cap C$ respectively. We remark that the convergence \eqref{eq step 2} now reads as 
	\begin{equation}
		f_{k_n}\to f_\infty \quad \text{in} \quad C^{1,\beta}(B').
		\label{convergenza grafici}
	\end{equation}
	
	We now prove that the curvatures $H_{E^{k_n}_h}$ of the sequence $E^{k_n}_h$ are converging in $C^{0,\beta}$ to the curvature of $E^\infty_h$ in the following sense
	\begin{equation}
		\label{goal step 2}
		H_{E^{k_n}_h}(\cdot,f_{k_n}(\cdot))\to H_{E^{\infty}_h}(\cdot,f_{\infty}(\cdot))\quad \text{in} \quad C^{0,\beta}(B').  
	\end{equation}
	We will follow an argument used in Step~3 of the proof of \cite[Theorem 4.3]{AFM}.
	
	Since we described $\bd E^{k_n}_h\cap C$ as a graph, the following formula for the curvature of $\bd E_h^{k_n}$ holds	
	\begin{equation}
		\label{formula curvatura}
		\div\left(  \dfrac{\nabla f_{k_n}(\cdot) }{\sqrt{1+|\nabla f_{k_n}(\cdot)|^2}} \right)=H_{E^{k_n}_h}(\cdot,f_{k_n}(\cdot))\quad \text{on} \quad B'
	\end{equation}
	and an analogous formula holds for $ \bd E^\infty_h$. From \eqref{formula curvatura} and the Euler-Lagrange equation \eqref{euler-lagrange equation}, by integrating on $B'$, we then obtain
	\begin{align}
		\label{align step 3}
		\lambda_{k_n}\mathcal H^{N-1}(B')&-\dfrac1h\int_{B'} sd_{E^{k_n-1}_h}(x',f_{k_n}(x'))\udH(x')\\[1.2ex]
		&=\int_{B'} H_{E^{k_n}_h}(x',f_{k_n}(x'))\udH(x')\nonumber \\[1.2ex]
		&=\int_{B'} \div\left(  \dfrac{\nabla f_{k_n}(x')}{\sqrt{1+|\nabla f_{k_n}(x')|^2}}   \right)\udH(x')\nonumber\\[1.2ex]
		&=\int_{\bd B'}  \dfrac{\nabla f_{k_n}(y)}{\sqrt{1+|\nabla f_{k_n}(y)|^2}}  \cdot y\,\ud\mathcal H^{N-2}(y)\nonumber,
	\end{align}
	where we set $y=x'/|x'|$ and integrated by parts in the last line. We can then exploit the  convergence \eqref{convergenza grafici} and the formula \eqref{formula curvatura} for the curvature of $E^\infty_h$  to prove 
	\begin{align*}
		\int_{\bd B'}  \dfrac{\nabla f_{k_n}(y)}{\sqrt{1+|\nabla f_{k_n}(y)|^2}}  \cdot y\,\ud\mathcal H^{N-2}(y)    &\to           \int_{\bd B'}  \dfrac{\nabla f_\infty(y)}{\sqrt{1+|\nabla f_\infty|^2}(y)}  \cdot y\,\ud\mathcal H^{N-2}(y)\\[1.2ex]
		&=\int_{B'} \div\left(  \dfrac{\nabla f_\infty(x')}{\sqrt{1+|\nabla f_\infty(x')|^2}}   \right)\udH(x')\\[1.2ex]
		&=H_{E^\infty_h}\,\mathcal H^{N-1}(B').
	\end{align*}
	Now, Lemma \ref{lemma 3.5n} ensures that $sd_{E^{k_n-1}_h}\to sd_{E^\infty_h}$ uniformly and we can use the convergence \eqref{convergenza grafici} to obtain 
	\[sd_{E^{k_n-1}_h}((\cdot,f_{k_n}(\cdot)))  \to    sd_{E^{\infty}_h}((\cdot,f_{\infty}(\cdot)))=0 \quad \text{uniformly on }B',\]
	since $\bd E^\infty_h\cap C=\{(x',f_\infty(x')): x'\in B')\}$ by definition. Therefore we find
	\[\int_{B'} sd_{E^{k_n-1}_h}((x',f_{k_n}(x'))) \udH(x')\to \int_{B'} sd_{E^{\infty}_h}((x',f_{\infty}(x'))) \udH(x')=0. \]
	We then conclude that  \eqref{align step 3} converges to $H_{E^\infty_h}\mathcal H^{N-1}(B')$ and thus it must hold
	\[ \lambda_{k_n}\to H_{E^\infty_h}.  \]
	From \eqref{euler-lagrange equation}, the previous result and the fact that the signed distance functions are all equi-lipschitz, we conclude that for any $ \beta\in (0,1)$, the sequence $(H_{E^{k_n}_h}(\cdot,f_{k_n}(\cdot)))$ is bounded in $C^{0,\beta}(B')$ and thus it converges uniformly to $H_{E^\infty_h}(\cdot,f_{\infty}(\cdot))$. This proves the convergence \eqref{goal step 2}. 
	
	We remark that the previous result also hold if we describe the sets of the flow $E_h^{k_n}$ as normal deformations of $E^\infty_h$, that is there exist functions $\varphi_{k_n}: \bd E^\infty_h\to \R$ such that $ E_h^{k_n}=(E_h^{\infty})_{\varphi_{k_n}}$. In this case the convergence \eqref{eq step 2} reads as
	\[ \varphi_{k_n}\to 0\quad \text{in} \quad C^{1,\beta}(\bd E^\infty_h),  \]
	and this and Lemma \ref{lemma 3.5n} ensure that
	\[ sd_{E^{k_n-1}_h}(\cdot+\varphi_{k_n}(\cdot)\nu_{E^\infty_h}(\cdot))  \to    sd_{E^{\infty}_h}(\cdot)=0 \quad \text{uniformly on }\bd E^\infty_h.  \]
	Now, the convergence of the curvatures reads as 
	\[  H_{E^{k_n}_h}(\cdot +\varphi_{k_n}(\cdot)\nu_{E^\infty_h}(\cdot))\to H_{E^\infty_h}(\cdot)\quad \text{in} \quad C^{0,\beta}(\bd E^\infty_h). \]
	We can then apply directly \cite[Lemma 7.2]{AFM} to obtain that the subsequence $E^{k_n}_h$ is converging to $E^\infty_h$ in $C^{2,\beta}$.
	
	To prove the $C^{2,\beta}-$closeness of the limit point we argue by contradiction. Assume that a sequence of limit points $(E^{\infty,l}_{h})_{l\in\N}$ is converging in $C^{1,\beta}$ to $E$ but there exists $\sigma>0$ such that 
	\[ \dist_{C^{2,\beta}}(E,E^{\infty,l}_h)>\sigma  \]
	for every $l$ large enough. Again, we describe locally $\bd E^{\infty,l}_h$ and $\bd E$ as graphs of suitable functions $f_{\infty,l},f : B'\to (-L,L)$  and we can repeat the same argument previously employed to prove that 
	\[  H_{E^{\infty,l}_h}((\cdot,f_{\infty,l}(\cdot)))\to H_E((\cdot,f(\cdot)))\qquad \text{in }C^{0,\beta}(B').  \] 
	This time the argument is simpler, since the limit points are stationary sets for the perimeter and thus  their Euler-Lagrange equation  is
	\[ H_{E^{\infty,l}_h}=\lambda_{E^{\infty,l}_h}\in\R\quad \text{on} \quad \bd E^{\infty,l}_h.  \]
	Again, Lemma 7.2 in \cite{AFM} yields the desired contradiction.\\
	\textbf{Step 3.} (Uniqueness up to translations and $C^k$ convergence) By the previous step we can find a suitable function $\varphi_\infty\in C^{2,\beta}(\bd E)$ such that  $E^\infty_h=E_{\varphi_\infty}$. Up to introducing a further translation given by Lemma \ref{lemma 3.8_AFM}, the hypotheses of Theorem \ref{teo alex} are satisfied and thus 
	\[ \|\varphi_\infty\|_{H^1(\bd E)}\le C\|H_{E^\infty_h}-\overline H_{E^\infty_h}\|_{L^2(\bd E)}=0, \]
	since the set $E^\infty_h$ is a stationary set for the perimeter. Therefore $E^\infty_h$ is a translated of the set $E$.
	
	A standard bootstrap method based on the elliptic regularity theory combined with the Euler-Lagrange equation \eqref{euler-lagrange equation} yields the convergence in $C^k$ for every $k\in\N$.
\end{proof}


\subsection{Exponential convergence of the whole flow}
\label{sec convergence}
In this subsection we will prove that the translations introduced in Proposition \ref{teorema convergenza a meno di traslazioni} decay to zero exponentially fast. In order to prove this result we will estimate the decay of the dissipations via a dissipation-dissipation inequality, which in turn relies on the quantitative Alexandrov type estimate established in Theorem \ref{teo alex}. We start by recalling some preliminary results from \cite{MPS}.

The following lemma is an adaptation to our case of \cite[Lemma~3.8]{MPS}. Its proof can be found in the Appendix.
\begin{lemma}[\textit{A priori} estimates]
	\label{lemma a priori estimates}
	Let $\eta>0$ and let $E \subset \T^N$ be a strictly stable set. There exists $\delta>0$ with the following property: if $f_1, f_2\in C^1(\bd E)$ with $\|f_i\|_{C^1(\bd E)}\le \delta$ and $|E_{f_i}|=|E|$ for $i=1,2$ we have
	\begin{align}
		C_1(1-\eta)\|f_1-f_2\|_{L^2}^2\le &\mathcal D(E_{f_1},E_{f_2})\le C_1 (1+\eta)\|f_1-f_2\|_{L^2}^2\label{stima dissipazione norma 2}\\[2ex]
		\dfrac{1-\eta }2\int_{\bd \rosso{E_{f_1}}}\sd^2_{ E_{f_2}}\udH\le &\mathcal D(E_{f_1},E_{f_2})\le \dfrac{1+\eta }2\int_{\bd \rosso{E_{f_1}}}\sd^2_{ E_{f_2}}\udH\label{stima dissipazione distanza^2}\\[2ex]
		|\textnormal{bar}(E_{f_1})-\textnormal{bar}(E_{f_2})|^2\le& C_2\|f_1-f_2\|^2_{L^2} \le \frac{C_2}{C_1(1-\eta)}\mathcal D(E_{f_1},E_{f_2})
		\label{stima baricentri}
	\end{align}
	for suitable constants $C_1,\, C_2>0$.
\end{lemma}

The following lemma  proves the crucial dissipation-dissipation inequality \eqref{dissipation-dissipation inequality} (see \cite[Lemma 3.9]{MPS}). This result will play a central role in the proof of Theorem \ref{teorema convergenza esponenziale}. Its proof is based on the Alexandrov-type estimate contained in Theorem \ref{teo alex}.
\begin{lemma}
	Let $h>0$ and let $E \subset \T^N$ be a strictly stable set. There exist constants $C,$ $\delta>0$  with the following property: for any pair of normal deformations $E_{f_1},$ $E_{f_2}$ with $f_i\in C^2(\bd E),$ $\|f_i\|_{C^1(\bd E)}\le \delta$, and such that $|E_{f_2}|=|E|,$ $|\int_{\bd E}\nu_E f_2\ud \mathcal H^{N-1}|\le \delta\|f_2\|_{L^2(\bd E)}$ and
	\begin{equation}
		H_{E_{f_2}}+\dfrac {\sd_{ E_{f_1}}}h=\lambda \quad \text{on}\quad \bd E_{f_2}
		\label{diss-diss euler lagrange}
	\end{equation}
	for some $\lambda\in\R$, we have
	\begin{equation}
		\mathcal D(E,E_{f_2})\le C\mathcal D(E_{f_2},E_{f_1}).
		\label{dissipation-dissipation inequality}
	\end{equation}
	\label{lemma 3.5}
\end{lemma}

\rosso{
\begin{proof}
    By Theorem \ref{teo alex}, for $\delta$ sufficiently small, we get
    \begin{equation*}
        \begin{split}
            \|f_2\|_{L^2(\bd E)}^2 &\le C \|H_{E_{f_2}}-\overline{H}_{E_{f_2}}\|_{L^2(\bd E)}^2\le C\|H_{E_{f_2}}-\lambda \|_{L^2(\bd E)}^2\\
            &\le 2 C \|H_{E_{f_2}}-\lambda \|_{L^2(\bd E_{f_2})}^2= \frac{2C}{h^2} \int_{\bd E_{f_2}} \textnormal{sd}_{E_{f_1}}^2 \udH,
        \end{split}
    \end{equation*}
    where the third inequality follows by bounding the Jacobian of the change of variables by $2$ (see \eqref{Jac_gen}). By combining the previous inequalities with \eqref{stima dissipazione norma 2} and \eqref{stima dissipazione distanza^2}, we obtain the thesis.
\end{proof}
}

We are now able to prove our main result.
The proof relies on our previous result Proposition~\ref{teorema convergenza a meno di traslazioni}, however this time we have to show that the translations introduced converge to an appropriate translation $\xi$. To achieve this result, we will obtain in Step 1 some estimates on the dissipations along the flow by comparing the energy with a suitable competitor. Once the (exponential) decay of the dissipations is proved, the convergence of the translations follows (see Step~2). The last step is devoted to prove the exponential convergence of the sets.

\begin{proof}[Proof of Theorem \ref{teorema convergenza esponenziale}]
	Let $h^*>0$, $\delta^*>0$ and $(\tau_n)_{n \in\N}$ be given by Proposition \ref{teorema convergenza a meno di traslazioni}. Fix $h<h^*$ and set $E_n:=E_h^n$. We split the proof in three steps.\\
	\textbf{Step 1.} (Exponential convergence of dissipations) Testing the minimality of $E_n$ with $E_{n-1}$ we obtain
	\[ P(E_n)+\dfrac 1h\mathcal D(E_n,E_{n-1})\le P(E_{n-1}). \]
	Recalling that $P(E_n)\to P(E)$ and summing the previous inequality from $n+1$ to $+\infty$ we get
	\begin{equation}
		\sum_{k=n+1}^{+\infty} \dfrac 1h\mathcal D(E_k,E_{k-1})\le  P(E_n)-P(E).
		\label{confronto facile2}
	\end{equation}
	
	We will now construct a suitable competitor to estimate the dissipation at the step $n-1$ with the difference of perimeters. 
	Since, by Proposition \ref{teorema convergenza a meno di traslazioni}, we have
	\begin{equation}
		E_n+\tau_n\to E   \quad \text{in} \quad C^k \quad \forall k \in \N,
		\label{convergenza}
	\end{equation}
	the translated sets of the flow, for $n$ large enough, can be written as normal deformations of the set $E$, that is there exists $g_n : \bd E\to \R$ such that
	\[ E_n+\tau_n=E_{g_n}, \]
	where $E_{g_n}$ was defined in \eqref{normal_deformation}.
	The convergence \eqref{convergenza} then reads as $g_n\to 0$ in $C^k$ as $n\to \infty$. Let $\sigma_n$ be the translations introduced by Lemma \ref{lemma 3.8_AFM} with $E_n +\tau_n$ instead of $F$. From the convergence in $C^k$ of $E_n+\tau_n$ to $E$, we deduce that $\sigma_n \to 0$ as $n\to\infty$. Therefore, setting 
	\[F_n:=E_n+\tau_n+\sigma_n,\]
	we have that $F_n\to E$ in $C^k$  and $F_n=E_{f_n}$ with $f_n:\bd E\to\R$ satisfying
	\[ \left|\int_{\bd E} f_n\nu_E \udH\right|\le \delta\|f_n\|_{L^2(\bd E)}\quad \text{and}\quad \|f_n\|_{W^{2,p}(\bd E)}\le C \|g_n\|_{W^{2,p}(\bd E)}   \]
	for $p>N-1$. Consider now the competitor
	\[ \mathscr E_n:=E-\tau_{n-1}-\sigma_{n-1}. \]
	From the minimality of $E_n$ we easily deduce 
	\begin{equation}
		P(E_n)+\dfrac 1h \mathcal D(E_n, E_{n-1})\le  P(\mathscr E_n)+\dfrac 1h \mathcal D(\mathscr E_n,E_{n-1})=P(E)+\dfrac 1h\mathcal D(E,E_{n-1}+\tau_{n-1}+\sigma_{n-1})
		\label{confronto competitor}
	\end{equation}
	where we used the translational invariance of the dissipations. From Lemma \ref{lemma 3.5n} we obtain that the sequence  $E_{n-2}+\tau_{n-1}+\sigma_{n-1}$ converges in $C^k$ to the same limit of $E_{n-1}+\tau_{n-1}+\sigma_{n-1}$, that is to $E$. In particular, for $n$ large enough we can write $E_{n-2}+\tau_{n-1}+\sigma_{n-1}=E_\psi$ for a suitable function $\psi:\bd E\to \R$ (depending on $n$) and with $\|\psi\|_{C^1(\bd E)}$ small. From Lemma \ref{lemma 3.5} we can then estimate the right hand side of \eqref{confronto competitor} with
	\begin{align*}
		\mathcal D(E,E_{n-1}+\tau_{n-1}+\sigma_{n-1})=&\mathcal D(E,F_{n-1})=\mathcal D(E,E_{f_{n-1}})
		\le C\mathcal D(E_{f_{n-1}}, E_\psi)\\
		=&C\mathcal D(E_{n-1}+\tau_{n-1}+\sigma_{n-1},E_{n-2}+\tau_{n-1}+\sigma_{n-1})\\
		=&C\mathcal D(E_{n-1},E_{n-2}).
	\end{align*}
	From the previous inequality and \eqref{confronto competitor} we obtain
	\begin{equation}
		P(E_n)-P(E)=P(E_n)-P(\mathscr E_n)\le \dfrac Ch \mathcal D(E_{n-1},E_{n-2}).
		\label{risultato confronto competitor}
	\end{equation}
	Now, \eqref{confronto facile2} and \eqref{risultato confronto competitor} yield
	\begin{align*}
		\sum_{k=n-1}^{\infty} \dfrac 1h\mathcal D(E_k,E_{k-1})
		=& \sum_{k= n+1}^{\infty} \dfrac 1h\mathcal D(E_k,E_{k-1}) + \dfrac 1h\mathcal D(E_n,E_{n-1})+\dfrac 1h\mathcal D(E_{n-1},E_{n-2})\\
		\le& \dfrac {C+1}h \mathcal D(E_{n-1},E_{n-2})+\dfrac 1h\mathcal D(E_n,E_{n-1})\\
		\le& \dfrac{C+1}h \left( \mathcal D(E_{n-1},E_{n-2})+\mathcal D(E_n,E_{n-1})  \right).
	\end{align*}
	We then apply Lemma \ref{lemma 3.10} below (with $l=2$) to conclude 
	\begin{equation}
		\mathcal D(E_n,E_{n-1})\le \left( 1-\dfrac 1{C+1} \right)^{n/2} \left(  P(E_0)- P(E) \right).
		\label{convergenza dissipazioni}
	\end{equation}
	\\
	\textbf{Step 2.} (Exponential convergence of barycenters)
	Set 
	\begin{equation}
		b= \left( 1-\dfrac 1{C+1} \right)^{\frac 14}\in(0,1).
		\label{def b}
	\end{equation} 
	From \eqref{convergenza}  and Lemma \ref{lemma 3.5n} both the sequences $(E_n+\tau_n)_{n\in\N}$ and $(E_{n-1}+\tau_n)_{n\in\N}$ converge in $C^k$ to $E$. Therefore, for $n$ large enough, there exist some functions $f_{1,n},\, f_{2,n}\in C^k(\bd E)$ such that 
	\[ E_{n}+\tau_n=E_{f_{1,n}},\quad E_{n-1}+\tau_n=E_{f_{2,n}} \]
	and $\|f_{i,n}\|_{C^k(\bd E)}\to 0$ as $n\to \infty$ for $i=1,2$. From \eqref{stima baricentri} and \eqref{convergenza dissipazioni}  we can estimate for $n$ sufficiently large
	\begin{align*}
		|\text{bar}(E_n)- \text{bar}(E_{n-1})|&=|\text{bar}(E_n+\tau_n)- \text{bar}(E_{n-1}+\tau_n)|\\
		&=|\text{bar}(E_{f_{1,n}})- \text{bar}(E_{f_{2,n}})|\\
		&\le C\sqrt{\mathcal D(E_{f_{1,n}},E_{f_{2,n}})}=\sqrt{\mathcal D(E_n,E_{n-1})}\\
		&\le C\left( P(E_0)-P(E)  \right)^{1/2}b^n.
	\end{align*}
	In turn, the above estimate implies that $(\text{bar}(E_n))_{n\in\N}$ satisfies the Cauchy condition, thus the whole sequence admits a limit $\bar \xi\in\T^N$. Moreover, the convergence is exponentially fast in the sense that 
	\[ |\text{bar}(E_{f_{1,n}})-\bar \xi|\le \sum_{k=n+1}^\infty |\text{bar}(E_{f_{1,n}})-\text{bar}(E_{f_{2,n}})|\le C\left( P(E_0)-P(E)  \right)^{1/2}\dfrac {b^n}{1-b}  \]
	for $n$ large enough. 
	Recalling \eqref{convergenza} we thus conclude that there exists a suitable translation $\xi \in \T^N$ such that for every $k\in\N$  
	\[ E_n\to E-\xi \quad \text{in }C^k \quad \text{as }n\to\infty.\]
	\textbf{Step 3.} (Exponential convergence of the sets) By the previous step we can write, for $n$ large, the boundaries of the evolving sets as radial graphs of the limit set $E-\xi$. Precisely, for $n$ large enough there exist functions $f_n$ such that
	\begin{equation}
		E_n+\xi=E_{f_n}\quad \text{and}\quad \|f_n\|_{C^k(\bd E)}\to 0 \quad \text{as }n\to\infty.
		\label{convergenza deformazioni}
	\end{equation}
	From \eqref{stima dissipazione norma 2} and for $n$ large enough we have $\|f_n-f_{n-1}\|_{L^2(\bd E)}\le 2\sqrt{\mathcal D(E_n,E_{n-1})}$ and thus, recalling \eqref{convergenza dissipazioni} and arguing as in Step 2, we get
	\begin{equation}
		\|f_n\|_{L^2(\bd E)}\le \sum_{k=n+1}^\infty \|f_n-f_{n-1}\|_{L^2(\bd E)}\le \left(  P(E_0)-P(E)  \right)^{1/2} \dfrac {b^n}{1-b}
		\label{velocita convergenza}
	\end{equation}
	where $b$ is as in \eqref{def b}. The above estimate yields the exponential decay of the $L^2-$norms of the radial graphs. We recall the well-known Gagliardo-Nieremberg inequality: for every $j\in\N$ there exists $C>0$ such that, if $g$ is smooth enough on the boundary of a smooth set $E$, then 
	\begin{equation}
		\| D^k g \|_{L^2(\bd E)}\le C\| D^{2k}g\|^{1/2}_{L^2(\bd E)}\|g\|^{1/2}_{L^2(\bd E)} 
		\label{gagliardo-nieremberg}
	\end{equation}
	where $D^k$ stands for the collection of all the $k-$th order derivatives of $g$, see e.g. \cite[Theorem~3.70]{Aub}.
	Now, by \eqref{convergenza deformazioni} for every $k$ there exists $n_k$ such that $\sup_{n\ge n_k}\| D^{2k} f_n\|_{L^2(\bd E)}\le 1$, therefore we may apply \eqref{gagliardo-nieremberg} to $f_n$ to deduce from \eqref{velocita convergenza} that also $\| D^k f_n\|_{L^2(\bd E)}$ decays exponentially fast for all $k\in\N$. The Sobolev immersion Theorem then yields the exponential decay in $C^k$ for every $k$ thus completing the proof of the result.
\end{proof}

\begin{lemma}
	Let $(a_n)_{n \in \N}$ be a sequence of non-negative numbers. Assume furthermore that there exist $c>1,\ l\in\N$ such that $\sum_{n=k}^\infty a_n\le c\sum_{j=k}^{k+l-1}a_j$ for every $k\in\N$. Then 
	\[ a_k\le \left(  1+\dfrac 1c \right)^{\frac kl}S \]
	for every $k\in\N$, where  $S=\sum_{n=1}^\infty a_n.$ 
	\label{lemma 3.10}
\end{lemma}
The proof of the previous lemma can be found in \cite[Lemma 3.11]{MPS}.


\section{Two-dimensional case}
\label{Two-dimensional case}

In this section, we completely characterize the long-time behaviour of the discrete flow in dimension two.
This particular choice for the dimension is purely technical and can be justified as follows. In the two-dimensional flat torus we have a complete characterization of the critical points of the perimeter: they consist in unions of disjoint discs (having the same area) or in unions of disjoint lamellae (possibly having different areas), or their complements. It turns out that these sets are all strictly stable. This allows us to conclude that either the connected components of any limit point of the discrete flow or the ones of their complements are strictly stable sets. We remark that in higher dimension this could not be true anymore.

Fix $h$, $m>0$ and let $(E^n_h)_{n\in \N}$ be a flow with initial set $E_0 \subset \T^2$ such that $|E_0|=m$. We recall that, by Proposition \ref{proprieta flusso}, there exists $s_0>0$ such that the distance between the connected components of the set $E_h^n$ is at least $s_0$. Moreover, the proposition also provides a bound from below on the diameter of the connected components. Set
\[P_\infty:=\lim_n P(E_h^n)\]
as the limit of the monotone sequence of the perimeters along the discrete flow. 
Let $F$ be any possible limit point of the sequence $(E_n^h)_{n \in \N}$. We observe that if $F$ is a union of discs then its number of connected components must be $\pi^{-1} P_{\infty}^2/(4m)$ and therefore the form of the limit point is uniquely determinated up to translations. Analogously, if $F$ is the complement of a union of discs, $F^c$ is made of $\pi^{-1} P_\infty^2/(4-4m)$ connected components and thus it is uniquely determinated up to translations of its complement. In the case when $F$ is a union of lamellae the number of connected components is\rosso{, in general, less than or equal to} $P_\infty/2$, and we have no information on the area of the single components.

Since we will consider $h$ as a fixed parameter, from now on we will denote by $E_n$ the set $E_h^n$.

\begin{oss}[Remarks on the uniform $C^{1,\alpha}-$closeness to limit points]
	\label{remark scrittura lamelle}
	We remark that for every $ \varepsilon>0$ there exists $ n_0=n_0(\varepsilon)\in \N$ such that for every $n\ge n_0 $ it holds 
	\begin{equation}
		\label{scrittura insiemi flusso}
		|E_n\triangle \bigcup_{i=1}^{l_n} F_{i,n} |\le \varepsilon \quad \text{or} \quad |E_n^c\triangle \bigcup_{i=1}^{L_n} F_{i,n} |\le \varepsilon,
	\end{equation}
	where, in the first case, $\bigcup_{i=1}^{l_n} F_{i,n}$ is a union of disjoint lamellae or a union of disjoint discs, with $F_{i,n}$ having the same mass of the $i-$th connected component of $E_n$; $l_n$ is either \rosso{less than or equal to} $ P_\infty/2$ if $F_{i,n},\ i=1,\dots, l_n$, are lamellae or $l_n=\pi^{-1} P_{\infty}^2/(4m)$ if they are discs; in the second case, $\bigcup_{i=1}^{L_n} F_{i,n}$ is a union of disjoint discs, with $F_{i,n}$ having the same mass of the $i-$th connected component of $E_n^c$ and $L_n=\pi^{-1} P_{\infty}^2/(4-4m)$.
	This can be easily proved recalling that any subsequence of the flow admits a further subsequence converging in $L^1$ to a set of the aforementioned form.

	Moreover, the classical regularity theory of $\Lambda-$minimizers implies that the previous result can be improved. Consider, for the sake of simplicity, that $E_n$ satisfies the first inequality in \eqref{scrittura insiemi flusso}(the other case is analogous). Then one can prove that for every $ \varepsilon>0$ there exists $ n_0=n_0(\varepsilon)$ such that for every $n\ge n_0 $ it holds
	\begin{equation}
		E_n=\bigcup_{i=1}^{l_n} (F_{i,n})_{f_{i,n}} \quad \text{where} \quad f_{i,n}\in C^{1,\alpha}(\bd F_{i,n}), \ \|f_{i,n}\|_{C^{1,\alpha}(\bd F_{i,n})}\le \varepsilon.
		\label{scrittura lamelle}
	\end{equation}
\end{oss}

\begin{oss}\label{rmk lamellae}
    \rosso{In this remark, we identify $\T^2$ with the unit square $[0,1)^2.$
    We prove that for a fixed $M>0$ there exists a finite number of slopes such that, for any lamella $L$ having one of those slopes, we have $P(L)\le M$.}
    
    \rosso{Fix a lamella $L$. Let $a \subset \T^2$ be one of the two components of the boundary of $L$, and suppose that $(0,0) \in a$. Since $a$ is a closed curve in $\T^2$, by periodicity, the line in $\R^2$ passing through the origin and with the same slope of $a$ must  also pass through a point of the form $(p,q) \in \N \times \N$ with $p,q$ coprime or equal to $(0,1)$ or $(1,0)$.  We then remark that the length in $\T^2$ of $a$ is equal to the one of the segment between the origin and $(p,q)$, that is length$(a)=|(p,q)|.$ 
    \begin{figure}
        \centering
        \includegraphics[scale=0.4]{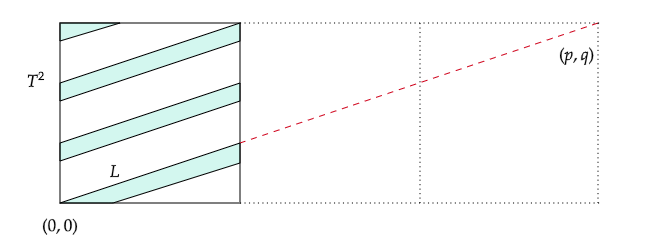}
        \caption{The lamella $L$ in light blue, the line $a$  dashed in red.}
    \end{figure}
    Since $P(L)=2\,\text{length}(a)$, in order to have $P(L) \le M$, the point $(p,q)$ must be contained in the disc of radius $M/2$.
    Our claim follows since in the disc of radius $M/2$ there is a finite number of points belonging to $\N \times \N$.}
\end{oss}

In the following lemma we characterize the geometric form of any limit point of the discrete flow.

\begin{lemma}[Uniqueness of the form of the limit]
	\label{limite unico}
	Fix $h$, $ m>0$ and an initial set $ E_0 \subset \T^2$ with mass $m$.
	Let $(E_n)_{n \in \N}$ be a discrete flow starting from $E_0$. 
	Then either one of the following holds:
	\begin{itemize}
		\item[i)] the limit points of the flow are disjoint unions of $l$ discs of total area $m$, where $l=\pi^{-1} (4m)^{-1}P_{\infty}^2$ belongs to $\N$,
		\item[ii)] the limit points of the flow are the complement of disjoint unions of $l$ discs of total area $1-m$, where $l=\pi^{-1} (4-4m)^{-1}P_{\infty}^2$ belongs to $\N$.
		\item[iii)] \rosso{the limit points of the flow are disjoint unions of $l$ lamellae of total area $m$, with the same slope and $l \le P_\infty/2$. Moreover, the equality $l=P_{\infty}/2 \in \N$ holds if and only if the limit is given by vertical or horizontal lamellae.}
	\end{itemize}
\end{lemma}

\begin{proof}
	We first employ a compactness argument and then use Lemma \ref{lemma 3.5n} to conclude. We start by fixing some notation. We denote by
	\begin{equation}
		\label{def unione palle}
		\mathscr E_B:= \bigcup_{i=1}^{l_B} B_i
	\end{equation}
	any disjoint union of $l_B=4^{-1}\pi m^{-1}P_{\infty}^2 $ discs each having radius $2m/P_{\infty}$; we denote by
	\begin{equation}
		\label{def unione palle compl}
		\mathscr E_{B^c}:= \left (\bigcup_{i=1}^{l_{B^c}} B_i\right )^c 
	\end{equation}
	the complement of any disjoint union of $l_{B^c}=4^{-1}\pi (1-m)^{-1}P_{\infty}^2 $ discs, each of radius $2(1-m)/P_{\infty}$; we denote by 
	\begin{equation}
		\label{def unione lamelle}
		\mathscr E_L:= \bigcup_{i=1}^{l_L} L_i
	\end{equation}
	any disjoint union of $l_L\rosso{\le\  }P_{\infty}/2$ lamellae \rosso{having the same slope} (\rosso{and} possibly having different masses). We remark that, for every fixed $P_\infty$ and $m$, the following holds
	\begin{equation}
		i:=\inf\{ \ud_H( \mathscr E_B, \mathscr E_L)\wedge \ud_H( \mathscr E_{B^c}, \mathscr E_L) \wedge \ud_H( \mathscr E_{B^c}, \mathscr E_B) \ :\ \mathscr E_L, \mathscr E_B, \mathscr E_{B^c} \text{ as above}  \}>0,
		\label{distanza palle-lamelle}
	\end{equation}
	\rosso{This is clear if we compare the families $\mathscr E_B, \mathscr E_{B^c}$ and a union of lamellae having the same slope. Since, by  Remark \ref{rmk lamellae}, there is a finite number of possible slopes for the lamellae, we conclude \eqref{distanza palle-lamelle}.} From Remark \ref{remark scrittura lamelle} the discrete flow is eventually $C^1-$close to a limit point of the form $\mathscr E_L, \mathscr E_B$ or $\mathscr E_{B^c}$.  Assume now by contradiction that the flow does not converge to a fixed configuration. Then, without loss of generality, we can assume that for every $0<\varepsilon<i/3$ there exist infinitely many indexes such that
	\[ \ud_H(E_{n-1},\mathscr E_B)\le \varepsilon \quad \text{and }\quad \ud_H(E_{n},\mathscr E_L)\le \varepsilon.\]
	Therefore we get
	\[ \ud_H(\mathscr E_B,\mathscr E_L)\le \ud_H(\mathscr E_B, E_{n-1})+\ud_H(\mathscr E_L,E_n)+\ud_H(E_n,E_{n-1})\le 2\varepsilon+\ud_H(E_n,E_{n-1}).  \]
	To reach the contradiction (compare \eqref{distanza palle-lamelle}), it is enough to show that  for every $ \varepsilon>0$ there exists $  n_0=n_0(\varepsilon)$ such that for every $n\ge n_0 $ it holds
	\begin{equation}
		\ud_H (E_{n-1},E_n)\le \varepsilon.
		\label{lemma 3.5 implica}
	\end{equation}
	Assume by contradiction the existence of a subsequence $n_k$ along which the flow satisfies
	\[\ud_H (E_{n_k-1},E_{n_k})> \varepsilon.\]
	Up to a further subsequence, $E_{n_k}\to F$, with $F$ being a set of the form $\mathscr E_B, \mathscr E_L $ or $\mathscr E_{B^c}$. But then Lemma \ref{lemma 3.5n} implies $ \sd_{E_{n_k-1}}\to \sd_F$ uniformly, which is clearly a contradiction.

 \rosso{Finally, we observe that in case $iii)$ the number of connected component is given by $\frac{P_{\infty}}{2|(p,q)|}$, where we used the same notation of Remark \ref{rmk lamellae}. Thus, $l=P_\infty/2$ if and only if $(p,q)$ is equal to $(0,1)$ or to $(1,0)$
 that means that the lamella is either vertical or horizontal.}
\end{proof}

Thanks to the previous lemma we can then conclude the proof of Theorem \ref{teorema convergenza dimensione 2},  the main result of this section.  While the proofs of assertions $i)$ and $ii)$ of Theorem \ref{teorema convergenza dimensione 2} are similar to the one of \cite[Theorem~3.4]{MPS}, the third one is slightly different, the main issue being that we can not fix the mass of the connected components of the limiting configuration. We will prove nonetheless the exponential convergence of the dissipations that, in turn, yields the convergence of the mass of the connected components of the flow. We start by a simple remark.

\begin{oss}[$C^{1,\alpha}$-closeness to lamellae]
	\label{C^1 closeness for lamellae}
	Let $\varepsilon>0$. Consider two lamellae $L_1, L_2$ \rosso{having the same slope,} possibly having different area and two $C^{1,\alpha}-$deformations $E_1, E_2$, respectively, of $L_1$ and $L_2$. 
	Suppose also that 
	\[\dist_{C^{1,\alpha}}(E_i,L_i)\le \varepsilon, \quad i=1,2. \]
	Then the closeness in $L^{\infty}$ of $E_1$ and $E_2$ implies that $E_2$ and $L_1$ are close in $C^{1,\alpha}$. Indeed, we first remark that
	\[\dist_{C^{1,\alpha}}(L_2,L_1)=\dist_{L^{\infty}}(L_2,L_1) \]
	since the components of the boundaries of $L_1$ and $L_2$ differ only by a translation.
	Moreover, the hypothesis $\dist_{L^\infty}(E_1,E_2)\le \varepsilon$ implies $\dist_{L^{\infty}}(L_2,L_1)\le 2\varepsilon$.
	Now, let $f_2$ be a suitable function such that $E_2=(L_2)_{f_2}$, then $\|f_2\|_{C^{1,\alpha}(\bd L_2)}\le \varepsilon$ and there exists a constant $|c|\le \dist_{L^{\infty}}(L_1,L_2)\le 2\varepsilon$ such that $E_2=(L_1)_{f_2+c}$.
	Therefore we obtain
	\begin{align*}
		\dist_{C^{1,\alpha}}(E_2,L_1)=\|f_2+c\|_{C^{1,\alpha}(\bd L_1)}\le \|f_2\|_{C^{1,\alpha}(\bd L_2)}+|c| \le \varepsilon+2\varepsilon= 3\varepsilon.
	\end{align*}
\end{oss}

\begin{proof}[Proof of Theorem \ref{teorema convergenza dimensione 2}]
	By Lemma \ref{limite unico}, we can assume that all the limit points of the flow are sets either of the form $\mathscr E_B$, $\mathscr E_{B^c}$ or $\mathscr E_L$ (see \eqref{def unione palle}, \eqref{def unione palle compl}, \eqref{def unione lamelle}).
	To conclude we need to prove that the whole sequence converges in $C^k$ and exponentially fast to a unique configuration.

	In the case when the limit points are of the form $\mathscr E_B$, the proof follows the same spirit of \cite[Theorem 3.4]{MPS}, but it is easier since we work in a compact space. The case when the limit points are of the form $\mathscr E_{B^c}$ is at all analogous: we simply remark that, if $F$ is a minimizer of \ref{def funzionale}, then its complement is a minimizer of the same problem with $E^c$ instead of $E$ and with $1-m$ instead of $m$. By studying the evolution of the complement of the discrete flow, we can conclude as before.

	Now, suppose that the limit points are of the form $\mathscr E_L$.
	We begin by observing that any subsequence of the flow admits a further subsequence converging in $L^1$ to a union of disjoint lamellae. Firstly, we prove the exponential decay of the dissipations. Testing the minimality of $E_n$ with $E_{n-1}$ we obtain
	\[ P(E_n)+\dfrac 1h\mathcal D(E_n,E_{n-1})\le P(E_{n-1}). \]
	Summing for $s\ge n+1$ we have
	\begin{equation}
		\sum_{s=n+1}^{+\infty} \dfrac 1h\mathcal D(E_s,E_{s-1})\le  P(E_n)-P_\infty.
		\label{confronto facile}
	\end{equation}
	With the notation previously introduced, for every $\varepsilon$ we can choose $n$ large enough such that \eqref{scrittura lamelle} holds. Let $F_{i,n}$ be the sets given by \eqref{scrittura lamelle}: by Lemma \ref{limite unico}, we know that $F_{i,n}$, $i=1, \ldots, l_n$, are eventually lamellae and  \rosso{$l_n=l\ge P_\infty/2$}.
	
	We will now construct a suitable competitor to estimate the dissipation at the step $n-1$ with the difference of the perimeters. For $n$ large enough consider the competitor
	\[ \mathscr L_n=\bigcup_{i=1}^{l}    F_{i,n-1}. \]
	We remark that, by definition and for $n$ large enough, this competitor has perimeter $P(\mathscr L_n)=P_\infty$. 
	By Proposition \ref{prop regularity discrete flow}, there exists $s_0=s_0(m,h,N, E_0)>0$ such that the connected components $E_{i,n}$ of $E_n$ satisfy
	\begin{equation*}
		\dist \left(E_{i,n}, E_{j,n}  \right)\ge s_0
	\end{equation*}
	for every $i\neq j$, moreover Remark \ref{remark scrittura lamelle} ensures that 
	\[ \dist \left( F_{i,{n-1}},F_{j,{n-1}}   \right)\ge \frac{s_0}2 \]
	holds for $n$ large enough and $i\neq j$. Thus, we can localize the dissipations
	\begin{align}
		\mathcal D(E_n,E_{n-1})&=\sum_{i=1}^l \mathcal D(E_{i,n},E_{i,n-1}),\label{eq 3.25-1}\\
		\mathcal D(\mathscr L_n,E_{n-1})&=\sum_{i=1}^l \mathcal D(F_{i,n-1},E_{i,n-1}).
		\nonumber
	\end{align}
	Testing the minimality of $E_n$ with $\mathscr L_n$ and using the previous equality we have
	\begin{equation}
		P(E_n)+\frac 1h \mathcal D(E_n,E_{n-1})\le P(\mathscr L_n) +\frac 1h \sum_{i=1}^l \mathcal D(F_{i,n-1}, E_{i,n-1}).
		\label{eq 3.26}
	\end{equation}
	Recalling Remark \ref{C^1 closeness for lamellae} and equations \eqref{scrittura lamelle} and \eqref{lemma 3.5 implica}, we then obtain that the connected components of both $E_{n-1}$ and $E_{n-2}$ are small normal $C^{1,\alpha}-$deformations of the connected components of $ \mathscr L_{n-1}$. Thus we can assume that both $E_{i,n-1}$ and $E_{i,n-2}$ can be described as normal deformation of $F_{i,n-1}$ for $i=1,\dots,k$. Let $f_{i,n-1}$ and $f_{i,n-2}$ be the functions (having small $C^{1,\alpha}-$norms) that describe respectively these deformations. 
	Now, recalling Lemma \ref{lemma 3.5}, we can estimate
	\begin{align*}
		\mathcal D(F_{i,n-1}, E_{i,n-1} )=& \mathcal D(F_{i,n-1},(F_{i,n-1})_{f_{i,n-1}}  )
		\le  C \mathcal D((F_{i,n-1})_{f_{i,n-1}}, (F_{i,n-1})_{f_{i,n-2}}  )\\
		=&C \mathcal D(E_{i,n-1}, E_{i,n-2} ).
	\end{align*}
	Thus, from equations  \eqref{eq 3.25-1} and \eqref{eq 3.26} we get 
	\[  P(E_n)-P_\infty= P(E_n)-P(\mathscr L_n)\le \frac Ch\sum_{i=1}^l \mathcal D(E_{i,n-1},E_{i,n-2})=\frac Ch \mathcal  D(E_{n-1},E_{n-2})\]
	and then \eqref{confronto facile} clearly yields
	\begin{align*}
		\sum_{s=n-1}^{\infty} \dfrac 1h\mathcal D(E_s,E_{s-1})
		=& \sum_{s= n+1}^{\infty} \dfrac 1h\mathcal D(E_s,E_{s-1}) + \dfrac 1h\mathcal D(E_{n-1},E_{n-2})+\dfrac 1h\mathcal D(E_{n},E_{n-1})\\
		&\le P(E_n)-P_\infty+\dfrac 1h\mathcal D(E_{n-1},E_{n-2})+\dfrac 1h\mathcal D(E_{n},E_{n-1})\\
		&\le \frac {C+1}h \mathcal D(E_{n-1},E_{n-2})+\dfrac 1h\mathcal D(E_{n},E_{n-1})\\
		&\le  \left(\frac {C+1}h \mathcal D(E_{n-1},E_{n-2})+\dfrac 1h\mathcal D(E_{n},E_{n-1}) \right).
	\end{align*}
	We can then conclude using the same arguments of \cite[Theorem 3.4]{MPS}. 
\end{proof}


\appendix
\section{Appendix}
We present the proof of Lemma \ref{lemma a priori estimates} for the reader's convenience. It is slightly different from the proof of \cite[Lemma 3.8]{MPS}.

\begin{proof}[Proof of Lemma \ref{lemma a priori estimates}]
	The proof of equations \eqref{stima dissipazione norma 2} and \eqref{stima dissipazione distanza^2} are quite analogous to the corresponding ones in \cite{MPS}. We recall it for the sake of completeness and to highlight the minor differences between the two versions.

	We start by observing that for any $\eta'>0$, if $\delta$ is sufficiently small, then for every $p_0\in\bd E_{f_2}$ the boundary of $E_{f_2}$ in a small disc must be contained in a cone 
	\begin{equation}
		\bd E_{f_2}\cap B_{4\bar\delta}(p_0)\subset G:=\left\lbrace  y\in\R^N\ :|(y-p_0)\cdot \nu_E(p_0)|^2\le \dfrac{\eta'^2}{1+\eta'^2} |y-p_0|^2  \right\rbrace. 
		\label{cono}
	\end{equation}
	We then divide the rest of the proof into two steps.\\
	\textbf{Step 1.} If $\delta$ is small enough , for every point $p=\lambda p_0\in B_{2\delta}(p_0)\ (\lambda>0),$ we have that 
	\[ \dfrac 1{1+\eta'}|p-p_0|\le \dist(p,\bd E_{f_2})\le |p-p_0|.  \]
	Indeed the second inequality is trivial by definition, since $p_0\in\bd E_{f_2}$. Concerning the first one, set $q\in\bd E_{f_2}$ such that $\dist(p,\bd E_{f_2})=|p-q|$, in particular $|p-q|\le |p-p_0|\le 2\delta$. From \eqref{cono} we infer that $q\in G$ and thus we have 
	\[  \dist(p,\bd E_{f_2})\ge \dist(p,G)=\dfrac 1{\sqrt{1+\eta'^2}}|p-p_0|\ge \dfrac 1{1+\eta'}|p-p_0| \]
	where we used the explicit formula for the projection of a point on a cone. If $p_0:=s+f_2(s)\nu_E(s)\in\bd E_{f_2}$ with $s\in \bd E$, we set
	\[ p_t:=p_0+t\dfrac {f_1(s)-f_2(s)}{|f_1(s)-f_2(s)|}\nu_E(s)\quad \text{for all } t\in [0,c|f_1(s)-f_2(s)|]  \]
	for an appropriate constant $c$ such that the quantities defined are regular. We deduce that 
	\begin{equation}
		\dfrac 1{1+\eta'}t\le \dist(p_t,\bd E_{f_2})\le t.
		\label{stima cose}
	\end{equation}
	Keeping the same notation and using the coarea formula (also recall \eqref{quiquo}), we infer that 
	\begin{align}
		\mathcal D(E_{f_1},E_{f_2})&=\int_{E_{f_1}\triangle E_{f_2}}\dist(x,\bd E_{f_2})\ud x\nonumber\\
		&=\int_{\bd E}\udH(s)\int_0^{c|f_1(s)-f_2(s)|} \dist(p_t, \bd E_{f_2})J\Phi(s,t) \ud t\nonumber \\
		&=\int_{\bd E}\udH(s)\int_0^{c|f_1(s)-f_2(s)|} \dist(p_t,\bd E_{f_2}) \ud t \nonumber\\ 
		&\quad+\int_{\bd E}\udH(s)\int_0^{c|f_1(s)-f_2(s)|}  \dist(p_t,\bd E_{f_2})(J\Phi(s,t)-1) \ud t.
		\label{stima cose 2}
	\end{align}
	Recalling that for every $s \in \bd E$ we have that $J\Phi(s,\cdot)-1$ is Lipschitz continuous with constant $H_E$, for $\delta $ small enough and using \eqref{stima cose}, we get
	\begin{align}
		\mathcal D(E_{f_1},E_{f_2})&\le (1+\delta H_E)\int_{\bd E} \udH(s) \int_0^{c|f_1(s)-f_2(s)|}t\ud t\\
		&=\dfrac {1+\delta H_E}2 c^2\int_{\bd E}|f_1(s)-f_2(s)|^2 \udH(s) ,   \label{stima cose 3}
	\end{align}
	from which the second inequality in \eqref{stima dissipazione norma 2} follows by taking  $\delta $ small enough. On the other hand, by \eqref{stima cose} we also have
	\begin{align}
		\mathcal D(E_{f_1},E_{f_2})&\ge \dfrac{1-\delta H_E}{1+\eta'}\int_{\bd E} \udH(s) \int_0^{c|f_1(s)-f_2(s)|}t\ud t\nonumber\\
		&=\dfrac {1-\delta H_E}{1+\eta'}c^2\int_{\bd E}|f_1(s)-f_2(s)|^2 \udH(s),
		\label{stima cose 4}
	\end{align}
	from which the first inequality in \eqref{stima dissipazione norma 2} follows by taking $\eta'$ and  $\delta $ small enough.\\
	\textbf{Step 2.} The inequalities \eqref{stima dissipazione distanza^2} and \eqref{stima baricentri} are now easy consequences. Indeed, by \eqref{stima cose} we have that, for every $p_1=(1+f_1(s))\nu_E(s)\in\bd E_{f_1}$, it holds
	\[ \dfrac c{1+\eta'}|f_1(s)-f_2(s)|\le \dist(p_1,E_{f_2})\le c|f_1(s)-f_2(s)|.  \]
	Therefore \eqref{stima dissipazione distanza^2} follows from \eqref{stima cose 3} and \eqref{stima cose 4}, by taking $\eta'$ and $\delta$ smaller if needed and using the same  change of coordinates used previously (recall that $J\Phi$ and its inverse are estimated from above by $1+C\delta$ for a suitable constant $C>0$).
	
	Finally, we prove \eqref{stima baricentri}. For $\delta$ small enough, we can bound the Jacobian by $2$ and therefore we obtain
	\begin{align*}
		|\text{bar}(E_{f_1})-\text{bar}(E_{f_2})|\,|E|&=\left| \int_{E_{f_1}\setminus E_{f_2}}x\ud x -\int_{E_{f_2}\setminus E_{f_1}}x \ud x \right|\\
		&=\bigg| \int_{\bd E\cap \{f_1>f_2 \}}\udH(s)\int_{f_2(s)} ^{f_1(s)} (s+t\nu_E(s))J\Phi(s)  \ud t  \\
		&\,\,\, - \int_{\bd E\cap \{f_1<f_2\} }\udH(s)\int_{f_1(s)} ^{f_2(s)} (s+t\nu_E(s))J\Phi(s) \ud t \bigg|\\
		&\le 2\left |\int_{\bd E} \big(2s+(|f_1(s)|+|f_2(s)|) \nu_E(s)\big)\,|f_1(s)-f_2(s)|\udH(s)\right|\\
		&    \le C\|f_1-f_2\|_{L^2}
	\end{align*}
	and the conclusion follows from \eqref{stima dissipazione norma 2}.
\end{proof}


\subsection*{Data availability statement}
Data sharing is not applicable to this article as no datasets were generated or analysed during the current study.

\newpage



\begin{thebibliography}{18}
	\addcontentsline{toc}{section}{References}
	
	
	\bibitem{AFM}
	{\sc E. Acerbi, N. Fusco, M. Morini}: {\it Minimality via second variation for a nonlocal isoperimetric problem}. {Comm. Math. Phys.} {\bf 322} (2013), no. 2, 515-557.
	
	
	\bibitem{ATW}
	{\sc F. Almgren, J. Taylor, L. Wang}: {\it Curvature-driven flows: a variational approach.} {SIAM J. Control Optim.} {\bf 31} (1993), no. 2, 387-438.
	
	
	\bibitem{AmT}
	{\sc L. Ambrosio, P. Tilli}: {Topics on analysis in metric spaces.} {\it Oxford Lecture Ser. Math. Appl., 25. Oxford University, Oxford 2004.}
	
	
	\bibitem{Aub}{\sc T. Aubin,} {Some nonlinear problems in Riemannian geometry.} {\it Springer Monographs in Mathematics. Springer-Verlag, Berlin 1998.}
	
	
	\bibitem{BCCN}
	{\sc G. Bellettini , V. Caselles , A. Chambolle , M. Novaga }: {\it The volume preserving crystalline mean curvature
		flow of convex sets in $\R^N$.} {J. Math. Pure Appl.} {\bf 92} (5) (2009), 499-527.
	
	
	\bibitem{CRCT}
	{W. Carter, A. Roosen, J. Cahn, J. Taylor}: {\it Shape evolution by surface diffusion and surface attachment
		limited kinetics on completely faceted surfaces.} {Acta Metallurgica et Materialia}, {\bf 43} (1995), pp. 4309–4323.
	
	
	\bibitem{CS}
	{\sc R. Choksi, P. Sternberg}: {\it On the first and second variations of a nonlocal isoperimetric problem.} {J. Reine Angew. Math} {\bf 611} (2007), 75-108. 
	
	
	\bibitem{CD}
	{\sc I. Capuzzo Dolcetta, S. Finzi Vita, R. March}: {\it Area-preserving curve-shortening flows: from phase separation to image processing.} {Interface Free Bound.} {\bf 4(4)}: 325-343, 2002. 
	\bibitem{ES}
	{\sc J. Escher, G. Simonett}:{\it The volume preserving mean curvature flow near spheres.} {Proc. Am. Math. Soc.} {\bf 126}	(1998), 2789-2796.
	
	
	\bibitem{Hui}
	{\sc G. Huisken}: {\it The volume preserving mean curvature flow.} {J. Rein. Angew. Math.} {\bf 382} (1987), 35-48.
	
	\bibitem{JMPS}
	{\sc V. Julin, M. Morini, M. Ponsiglione, E. Spadaro}: {\it The Asymptotics of the Area-Preserving Mean Curvature and the Mullins-Sekerka Flow in Two Dimensions.} {CVGMT Preprint}, {http://cvgmt.sns.it/paper/5399/}
	
	
	\bibitem{KiKU}
	{\sc I. Kim, D. Kwon}: {\it Volume preserving mean curvature flow for star-shaped sets.} {Comm. Partial Differential Equations}, {\bf 59} no. 81 (2020).
	
	
	\bibitem{KM}
	{\sc B. Krummel, F. Maggi}: {\it Isoperimetry with upper mean curvature bounds and sharp stability estimates.} {Calc. Var. Partial Differ. Equ.} {\bf 56} (2017), Article n. 53.
	
	\bibitem{LS}
	{\sc S. Luckhaus, T. Sturzenhecker}: {\it Implicit time discretization for the mean curvature flow equation}. {Calc. Var. Partial Differential Equations} {\bf 3} (1995), 253-271.
	
	
	\bibitem{Mag}
	{\sc F. Maggi}: {Sets of finite perimeter and geometric variational problems. An introduction to geometric measure
		theory.} {\it Cambridge Studies in Advanced Mathematics,} 135. Cambridge University Press, Cambridge, 2012.
	
	
	\bibitem{May01}
	{\sc U. F. Mayer:}{\it A singular example for the average mean curvature flow}. {Experimental Mathematics} {\bf10} (1) (2001), 103-107.
	
	
	\bibitem{MaSi}
	{\sc U. F. Mayer, G. Simonett}:{\it Self-intersections for the surface diffusion and the volume-preserving mean curvature flow}. {Differential and Integral Equations} {\bf 13} (79) (2000), 1189-1199.
	
	
	\bibitem{MoSo}
	{\sc J.-M. Morel, S. Solimini}: {\it Variational Methods in Image Segmentation.} Birkh\"{a}user Boston, Inc., Boston, MA {\bf 14} (1995), xvi+245.

 \bibitem{MPS}
	{\sc M. Morini, M. Ponsiglione, E. Spadaro}: {\it Long time behaviour of discrete volume-preserving mean-curvature flows}. {J. Reine Angew. Math.} {\bf 784} (2022), 27-51.
	
	\bibitem{MS}
	{\sc L. Mugnai, C. Seis}: {\it On the coarsening rates for attachment-limited kinetics.} { SIAM J. Math. Anal. } {\bf 45} (2013), 324-344.
	
	
	\bibitem{MSS}
	{\sc L. Mugnai, C. Seis, E. Spadaro}: {\it Global solutions to the volume-preserving mean-curvature flow.} { Calc. Var.} {\bf 55} (2016), Art. 18, 23.
	
	
	
	
	
	\bibitem{Nii}
	{\sc J. Niinikoski}:{\it Volume preserving mean curvature flows near strictly stable sets in the flat torus.} {J. Diff. Eq.} {\bf 276} (2020), 149-186.
	
	
\end{thebibliography}
\end{document}